\documentclass[a4paper,12pt]{amsart}
 \input{macros}
 \usepackage[margin=1.0in]{geometry}
 \setlist{ 
	listparindent=\parindent,
	parsep=0pt,
 }

 \numberwithin{equation}{section}

 \author[J. Baudin]{Jefferson Baudin}
 \address{\'Ecole Polytechnique F\'ed\'erale de Lausanne, Chair of Algebraic Geometry \newline 
 \indent MA C3 575 (Bâtiment MA), Station 8, CH-1015 Lausanne}
 \email{jefferson.baudin@epfl.ch}

 \author[K. Rülling]{Kay Rülling}
 \address{School of Mathematics and Natural Sciences, Bergische Universität Wuppertal, \newline 
 \indent Gaußstraße 20, 42119 Wuppertal, Germany}
 \email{ruelling@uni-wuppertal.de}

 \title[Cartier crystals and dual-constructible $p$-torsion complexes]{Cartier crystals and dual-constructible $p$-torsion complexes}
 
 \subjclass[2020]{14G17, 13A35, 14F20} 
 \keywords{Cartier modules, Riemann--Hilbert correspondence, flat and relatively perfect site}

 \AtBeginDocument{%
 \def\MR#1{}
 }

\begin{document}

\begin{abstract}
    We show that given a Noetherian and $F$-finite scheme, the fixed point functor on Cartier crystals is fully faithful to the category of sheaves on the flat and relatively perfect site defined by Kato, and its essential image consists of the subcategory of dual-constructible sheaves, which we introduce.
\end{abstract} 

\maketitle

%======================================================================================================
%======================================================================================================

\tableofcontents

\section{Introduction}

Over the complex numbers, the Riemann--Hilbert correspondence is an equivalence of categories between regular holonomic $\cD$-modules and constructible sheaves on the analytic topology \cite{Hotta_Takeuchi_Tanisaki_D_modules_perverse_sheaves_and_representation_theory}, generalizing the equivalence between vector bundles with connections and local systems given by taking horizontal sections. In positive characteristic, we now have different (equivalent) analogues of this result \cite{Emerton_Kisin_Riemann-Hilbert_correspondence, Ohkawa_RH_embeddable_varieties, Bockle_Pink_Cohomological_Theory_of_crystals_over_function_fields, Bhatt_Lurie_RH_corr_pos_char, Mathew_RH_via_perfect_site}. The analogues of $\cD$-modules in these cases are \emph{Frobenius modules}, namely pairs $(\cM, \tau_{\cM})$ where $\tau_{\cM}$ is a $p$-linear endomorphism of $\cM$. The Riemann--Hilbert correspondence, in this case, asserts that taking fixed points of the $p$-linear operator is fully faithful ``up to nilpotence'' (when restricted to coherent objects), and the image consists of constructible étale $\bF_p$-sheaves. This result has found numerous applications to commutative algebra and algebraic geometry \cite{Emerton_Kisin_Riemann-Hilbert_correspondence, Bockle_Pink_Cohomological_Theory_of_crystals_over_function_fields, Bhatt&Co_Applications_of_perverse_sheaves_in_commutative_algebra, Bhatt&Co_MMP_for_3folds_in_mixed_char}.

Nevertheless, there are other objects of interest in positive characteristic that admit ``Frobenius type structures'', such as \emph{Cartier modules}, introduced in \cite{Blickle_Bockle_Cartier_modules_finiteness_results} (see also \cite{Schedlmeier_Cartier_crystals_and_perverse_sheaves}). A Cartier module consists of a pair $(\cN, \kappa_{\cN})$ where $\kappa_{\cN}$ is a $p^{-1}$-linear  endomorphism of $\cN$ (the standard example is the Cartier operator on $\omega_X \coloneqq \det(\Omega_X^1)$, where $X$ is a smooth scheme over a perfect field $k$, see \cite{Cartier_une_nouvelle_operation_sur_les_formes_differentielles}). It is therefore natural to ask whether taking fixed points also retains information in this case. Unlike the case of Frobenius modules, the correct topological setup in this case is not the étale site, but the \emph{flat and relatively perfect site} (in short, FRP site), see \cite{Kato_Duality_theories_for_the_p_primary_etale_cohomology_I, Breen_Extensions_du_groupe_additif, Breen_Extensions_du_groupe_additif_sur_le_site_parfait}. Here is our main result:

\begin{thmletter}[{\autoref{Sol_fully_faithful}, \autoref{cor:Sol-CartierCrys}}]\label{main_thm}
    Let $X$ be a Noetherian and $F$-finite scheme. Then the functor of fixed points \[ \Sol \colon \Crys_X^C \longrightarrow \Sh(X_{\FRP}, \bF_p) \] is fully faithful, as well as its induced functor on derived categories
    \[ \Sol \colon D^b(\Crys_X^C) \to D^b(X_{\FRP}, \bF_p). \] The essential image consists of the category of dual-constructible complexes $D^b_{dc}(X_{\FRP}, \bF_p)$.
\end{thmletter}

The category $\Crys_X^C$ in the above theorem denotes the category of coherent Cartier modules ``up to nilpotence'' (see \autoref{def Cartier module} for a precise definition). The notion of a dual-constructible complex is introduced in \autoref{sec:essential-image}.
It is a complex of sheaves $\cF^\bullet$ of $\bF_p$-vector spaces on $X_{\FRP}$ with the property that for 
any non-empty closed immersion $Z\inj X$  there exists a regular dense open $Z_0\subset  Z$
such that $h^!\cF^{\bullet}$ has dual-lisse cohomology sheaves, where $h\colon Z_0\inj X$ is the immersion and a dual-lisse sheaf
on $Z_{0, \FRP}$ is a sheaf of the form $E\otimes_{\bF_p} \Sol(\Omega^d_{Z_0/\bF_p})$, 
where $d={\rk}(\Omega^1_{Z_0/\bF_p})$ and $E$ is a lisse sheaf of $\bF_p$-vector spaces. 
We refer the reader to \autoref{def_dual_constr_cpx} for more details.
We also have a version of this result for Cartier crystals over $W_nX$, as introduced in \cite{Baudin_Duality_between_Witt_Cartier_crystals_and_perverse_sheaves}.

If we actually restrict to the étale site, the functor $\Sol$ above is not full (\autoref{RH_not_okay_with_etake_site}). Nevertheless, it remains faithful (see \autoref{Sol_faithful_on_etale_and_FRP_site}). The reason why the étale site was required in the case of Frobenius modules was to ensure to have enough fixed points. In the case of Cartier modules however, the situation is drastically different, and there are \emph{many} fixed points (even in the Zariski site, see \autoref{suprising_Zariski_faithfulness}). The purpose of introducing the FRP site is therefore to reduce the amount of morphisms between the sheaves in order to achieve fullness.

The main ingredient of the proof is the following commutativity result:

\begin{thmletter}[{\autoref{main_commutativity_thm}, \autoref{thm:constr-dual}}]\label{intro_commutativity}
    Let $X$ be a Noetherian and $F$-finite scheme with a unit dualizing complex $\omega_X^{\bullet}$. Then the diagram 
    \[ \xymatrix{D^b(\Crys_X^F)^{op}\ar[r]^{\bD_G}\ar[d]^{\Sol} & D^b(\Crys_X^C)\ar[d]^{\Sol}\\
    D^b(X_{\FRP},\bF_p)^{op}\ar[r]^{\bD_{K}} & D^b(X_{\FRP}, \bF_p)}
    \] commutes. Furthermore, the functor $\bD_K$ induces an equivalence \[ D^b_c(X_{\et}, \bF_p)^{op} \xlongrightarrow{\cong} D^b_{dc}(X_{\FRP}, \bF_p).  \]
\end{thmletter}

In the above theorem, the functor $\bD_G$ is the ``Grothendieck duality type'' equivalence of categories induced by $\omega_X^{\bullet}$ (see \cite{Baudin_Duality_between_perverse_sheaves_and_Cartier_crystals, Baudin_Duality_between_Witt_Cartier_crystals_and_perverse_sheaves}), and $\bD_K \coloneqq \cR\HHom(-, \Sol(\omega_X^{\bullet}))$ is Kato's duality functor (see \cite{Kato_Duality_theories_for_the_p_primary_etale_cohomology_I, Kato_Duality_theories_for_p_primary_etale_coh_II}). Since $\bD_G$, $\bD_K$ and the $\Sol$ functor on Frobenius crystals are known to be fully faithful, \autoref{main_thm} follows from \autoref{intro_commutativity}. The proof of \autoref{intro_commutativity} goes through a dévissage argument, and the necessary theory is developed in \autoref{section_FRP_site}.

As a corollary of the above theorems and our theory of dual-constructible sheaves, we obtain surprisingly strong results about finiteness of higher direct images of top logarithmic forms (note that there is no proper assumption):

\begin{corletter}[{\autoref{cor:sep-base}}]{\label{intro_cor:sep-base}}
Let $K$ be a separably closed field which has a $p$-basis of length $e$.
Let $X$ be a $d$-dimensional regular connected separated $K$-scheme, and fix $n \geq 1$. Then for all $i \geq 0$ 
\[H^i(X_{\et}, W_n\Omega^{d+e}_{X_{\et},\log})= M_i \otimes_{\bZ/p^n\bZ} W_n\Omega^e_{K,\log},\]
for some finite $\bZ/p^n\bZ$-module $M_i$. 
\end{corletter}

In the above, $W_n\Omega^i_{\log}$ denotes logarithmic forms in the $i$-th piece of the de Rham--Witt complex of Bloch--Deligne--Illusie \cite{Illusie_Complexe_de_de_Rham_Witt_et_cohomologie_cristalline}.

\subsection{Acknowledgements}

We would like to thank Léo Navarro Chafloque, Jakub Witaszek and Bogdan Zavyalov for interesting discusions related to the content of this article. The first author was funded by grant \#20021/231484 from the
Swiss National Science Foundation.

\subsection{Notation}

\begin{enumerate}
    \item Throughout, we fix an integer $n \geq 1$, a prime number $p$, an integer $r \geq 1$, and we set $q \coloneqq p^r$. 
    %All schemes will be $\bF_q$-schemes.
    \item Given a scheme $X$ of positive characteristic, we denote by $W_n\cO_X$ the $n$-truncated Witt vectors.	
    The pair consisting of the underlying topological space of $X$ and the sheaf of rings $W_n\cO_X$ defines a scheme, which we denote by $W_nX$. If $f \colon X \to Y$ is a morphism of schemes in positive characteristic, then the induced morphism $W_nX \to W_nY$ will still be denoted by $f$. In particular, we will denote by $F\colon W_nX\to W_nX$ the morphism induced 
    by the absolute Frobenius on $X$.
    \item Given a finite morphism $f \colon X \to Y$ of Noetherian schemes, the right adjoint of $f_* \colon \QCoh_X \to \QCoh_Y$ will be denoted $f^{\flat} \colon \QCoh_Y \to \QCoh_X$ (see \stacksproj{08YP}), and we shall write $f^! \coloneqq Rf^{\flat}$. For example, if $f$ is a closed immersion with associated ideal $\cI \inc \cO_Y$, then $f^{\flat}(
    \cM)$ is the sheaf of $\cI$-torsion elements in $\cM$ seen as a quasi-coherent $\cO_X=\cO_Y/\cI$-module.
		
         More generally, given a separated morphism of finite type $f \colon X \to Y$ of Noetherian schemes, the functor $f^!$ is the one defined in \stacksproj{0A9Y}.

    \item Let $R$ be a ring, and set $X = \Spec(R)$. Given an $R$-module $M$, the associated quasi-coherent sheaf of $X$ will be called $\ttilde{M}$.
\end{enumerate}
%======================================================================================================

\section{The flat and relatively perfect site}\label{section_FRP_site}

In this section, we study the flat and relatively perfect site from \cite{Kato_Duality_theories_for_p_primary_etale_coh_II}. We will 
(re)define basic functors and prove key properties which we will use later on. 

\medskip

Throughout, let $X$ be a Noetherian $F$-finite $\bF_q$-scheme and let $\Lambda$ be a finite ring, which will serve as coefficients for our sheaves. The interesting case for us is $\Lambda = W_n(\bF_q)$.

\begin{definition}[{\cite[Definition 1.1]{Kato_Duality_theories_for_the_p_primary_etale_cohomology_I}},
{\cite[\S 2]{Kato_Duality_theories_for_p_primary_etale_coh_II}}]
    \begin{enumerate}
        \item We say that a morphism $Y \to X$ is \emph{relatively perfect} if the relative Frobenius $F_{Y/X}$ 
             is an isomorphism.
        \item The \emph{flat and relatively perfect site} (in short, FRP site) of $X$ has as underlying category 
        the full subcategory of all $X$-schemes which are flat and relatively perfect over $X$, together with the étale topology. This site will be denoted $X_{\FRP}$. 
    \end{enumerate}
\end{definition}

\begin{remark}\label{basic_remark}
    Throughout the article, we will implicitly use the following two results:
    \begin{enumerate}[label=(\alph*)]
        \item a morphism $f\colon Y\to X$ locally of finite presentation is relatively perfect if and only if it is étale (see \cite[Lemma 1.3]{Kato_Duality_theories_for_the_p_primary_etale_cohomology_I});
        \item\label{basic_remarkb} if $f \colon Y \to X$ is a flat and relatively perfect morphism, then for all $n \geq 1$, 
        the morphism $W_nY \to W_nX$ is also flat and the diagram 
        \[ \begin{tikzcd}
            W_nY \arrow[d, "f"'] \arrow[r, "F"] & W_nY \arrow[d, "f"] \\
            W_nX \arrow[r, "F"] & W_nX
        \end{tikzcd} \]
        is a pullback square (see \cite[Lemma 2]{Kato_A_generalization_of_local_class_field_theory}).
    \end{enumerate}
\end{remark}

The relevant examples of FRP sheaves will be induced by \'etale sheaves of $\Lambda$-modules and 
quasi-coherent sheaves as follows:

\begin{lemma}\label{adjunction_FRP_and_etale}
\begin{enumerate}[label=(\arabic*)]
 \item\label{adjunction_FRP_and_etale1} The inclusion functor $X_{\et}\to X_{\FRP}$ induces a morphism of topoi $\rho\colon \Sh(X_{\FRP})\to \Sh(X_{\et})$,
 such that 
 \begin{enumerate}[label=(\alph*)]
 \item\label{adjunction_FRP_and_etale1a} 
       $\rho_*$ and $\rho^{-1}$ are exact, $\rho_*\circ \rho^{-1}=\id$, and $\rho^{-1}$ is fully faithful;
 \item $\rho_*\colon \Sh(X_{\FRP})\to \Sh(X_{\et})$ is the restriction, i.e., $\rho_*\cG=\cG|_{X_{\et}}$;
 \item\label{adjunction_FRP_and_etale1c} $\rho^{-1}\colon \Sh(X_{\et})\to \Sh(X_{FRP})$ sends $\cF$ to the sheaf, which on an object 
       $f \colon Y\to X$ in $X_{\FRP}$ is given by $\cF_{\FRP}(Y) = \Gamma(Y, f^{-1}(\cF))$.
 \end{enumerate}
 \item \label{FRP_embeds_in_etale} The constant sheaf defined by $\Lambda$ on $X_{\et}$ (resp. $X_{\FRP}$) is by abuse of notation again denoted by 
    $\Lambda$  (so that $\rho^{-1}\Lambda=\Lambda$). Then $\rho$ from \ref{adjunction_FRP_and_etale1} induces 
    an adjoint pair between sheaves of $\Lambda$-modules
    \[(-)_{\FRP}\coloneqq\rho^{-1}\colon \Sh(X_{\et},\Lambda)\rightleftarrows \Sh(X_{\FRP},\Lambda)\colon (-)|_{X_{\et}}=\rho_*,\]
    which satisfies the same properties as in \ref{adjunction_FRP_and_etale1}\ref{adjunction_FRP_and_etale1a}. Moreover, this extends 
    to an adjoint pair between the derived categories of $\Lambda$-modules
  \[(-)_{\FRP}\colon  D(X_{\et},\Lambda)\leftrightarrows D(X_{FRP},\Lambda): (-)|_{X_{\et}}\]
  satisfying the same properties as in \ref{adjunction_FRP_and_etale1}\ref{adjunction_FRP_and_etale1a}.
 \item\label{QCoh-FRP} Denote by $W_n\cO_{X_{\FRP}}$ the sheaf on $X_{\FRP}$ given on an object $f\colon Y\to X$ in $X_{\FRP}$ by
      \[W_n\cO_{X_{\FRP}}(Y)= \Gamma(Y, W_n\cO_Y).\]
      By adjunction the identity $W_n\cO_{X_{\et}}\to \rho_*W_n\cO_{X_{\FRP}}$ induces a morphism of sheaves of rings
      $\rho^{-1}W_n\cO_{X_{\et}}\to W_n\cO_{X_{\FRP}}$, which gives rise to an adjoint pair between sheaves of
      $W_n\cO$-modules
\[ \rho^* \colon \Sh(X_{\et}, W_n\cO_{X_{\et}})\rightleftarrows \Sh(X_{\FRP},W_n\cO_{X_{\FRP}}):\rho_*,\]
which satisfies the analog properties of \ref{adjunction_FRP_and_etale1}\ref{adjunction_FRP_and_etale1a}. 
Moreover, this adjunction induces an equivalence between the categories of quasi-coherent sheaves (in the sense of \stacksproj{03DL})
and an equivalence between the derived categories of quasi-coherent sheaves
\[ \rho^* \colon D(\QCoh(X_{\et}, W_n\cO_{X_{\et}}))\overset{\simeq}{\rightleftarrows} 
D(\QCoh(X_{\FRP}, W_n\cO_{X_{\FRP}})): \rho_*.\]
\end{enumerate}
 \end{lemma}
\begin{proof}
The statements in \ref{adjunction_FRP_and_etale1} 
 follow from \stacksproj{00XU} and \cite[Proposition 1.1]{Kato_Duality_theories_for_p_primary_etale_coh_II}.
 The statements in \ref{FRP_embeds_in_etale} 
 follow from this and \stacksproj{03D7} (adjunction), \stacksproj{04JC} (exactness of $\rho^{-1}$)
 and \stacksproj{0FND} (derived adjunction).
 Similarly we get \ref{QCoh-FRP}, where we additionally use that $W_n\cO_{X_{\FRP}}$ is a sheaf by \'etale descent and 
 that $\rho^{-1}W_n\cO_{X_{et}}\to W_n\cO_{X_{\FRP}}$ is flat, by \autoref{basic_remark}.
 The statements about quasi-coherent sheaves follow from the arguments in \stacksproj{03DX}.
 \end{proof}

\begin{definition}\label{defn:FRP}
We denote by 
\[(-)_{\FRP}\colon \QCoh_{W_nX}\to \Sh(X_{\FRP}, W_n(\bF_q)), \quad \cF\mapsto \cF_{\FRP},\]
the composition 
\[\QCoh_{W_nX}\xrightarrow{\simeq} \QCoh(X_{\et}, W_n\cO_{X_{\et}})\xrightarrow{\simeq} \QCoh(X_{\FRP}, W_n\cO_{X_{\FRP}})\to \Sh(X_{\FRP}, W_n(\bF_q)),\]
where the first equivalence is the classical one (e.g. \stacksproj{03DX}), the second equivalence is the one from  \autoref{adjunction_FRP_and_etale}.\ref{QCoh-FRP}, and the last map is the forgetful functor. 
Concretely, the sections of $\cF_{\FRP}$ over $f\colon Y\to X$ in $X_{\FRP}$ are given by
\[\cF_{\FRP}(Y)=\Gamma(Y, W_n\cO_Y\otimes_{f^{-1}W_n\cO_X} f^{-1}\cF) = \Gamma(Y, f^*\cF).\]
As $(-)_{\FRP}$ is exact it extends to an exact functor between  derived categories
\[(-)_{\FRP}\colon D(\QCoh_{W_nX})\to D(X_{\FRP}, W_n(\bF_q)), \quad \cF^\bullet\mapsto \cF^\bullet_{\FRP}.\]
\end{definition}

%Let us now focus on functors on this site. Our reference for morphisms of sites and so on will be \stacksproj{00UZ}.

\begin{construction}[Direct images, inverse images and extension by zero]
    Let $f \colon Y \to X$ be a morphism of Noetherian and $F$-finite $\bF_q$-schemes. The inverse image functor $X_{\FRP} \to Y_{\FRP}$ is  continuous and
    induces a pair of adjoint functors 
    \[f^{-1}\colon \Sh(X_{\FRP}, \Lambda) \rightleftarrows \Sh(Y_{\FRP}, \Lambda)\colon f_*,\] 
    with $f^{-1}$ exact, by \stacksproj{00X6}. %and \cite[Proposition 1.1]{Kato_Duality_theories_for_p_primary_etale_coh_II}. 
    If $f \colon Y \to X$ is FRP, then we will denote by $f_! \colon \Sh(Y_{\FRP}, \Lambda) \to \Sh(X_{\FRP}, \Lambda)$ the left adjoint of $f^{-1}$ (see \stacksproj{03DI}).
\end{construction}

We will need the following result several times:

\begin{lemma}\label{pushforwards_of_qcoh_commute_with_FRP}
    Let $f \colon X \to Y$ be a morphism of Noetherian and $F$-finite $\bF_q$-schemes. 
    Then the diagram \[ \begin{tikzcd}
        D^+(\QCoh_{W_nX}) \arrow[rr, "Rf_*"] \arrow[d, "(-)_{\FRP}"] &  & D^+(\QCoh_{W_nY}) \arrow[d, "(-)_{\FRP}"] \\
        D^+(X_{\FRP}, W_n(\bF_q)) \arrow[rr, "Rf_*"]                              &  & D^+(Y_{\FRP}, W_n(\bF_q))
    \end{tikzcd} \] commutes.
\end{lemma}
\begin{proof}
    Since this diagram commutes at the non-derived level by flat base change (\stacksproj{02KH}) and $(-)_{\FRP}$ is exact, it is enough to show that for any injective object $\cI \in \QCoh_{W_nX}$ and $i > 0$, we have $R^if_*(\cI_{\FRP}) = 0$.
    
    Let $g \colon Y' \to Y$ be a flat and relatively perfect morphism, and consider the pullback diagram 
    \[ \begin{tikzcd}
        X' \arrow[r, "f'"] \arrow[d, "g'"] & Y' \arrow[d, "g"] \\
        X \arrow[r, "f"]             & Y.
    \end{tikzcd} \]
    Note that by definition of the FRP site, we have to show that $R^if'_{\et, *}\left((g')^*\cI)_{\et}\right) = 0$, for all $i > 0$. By \stacksproj{03P2}, it is enough to show the vanishing  $R^if'_*\left((g')^*\cI\right) = 0$ on $Y'_{\Zar}$, for all $i > 0$, which holds by flat base change and the injectivity of $\cI$.
\end{proof}

We will end this section by emphasizing the case of closed immersions. Throughout, fix a closed immersion $i \colon Z \to X$, and let $j \colon U \to X$ is the open immersion corresponding to the complement of $Z$. 

\begin{lemma}\label{etale_functors_and_FRP_functors_agree}
    Let $f \colon Y \to X$ be a morphism (resp. an étale morphism, proper morphism) of Noetherian $F$-finite $\bF_q$-schemes.
    Then for any $\cF \in \Sh(X_{\et}, \Lambda)$ (resp. $\cG \in \Sh(Y_{\et}, \Lambda)$), we have that 
    \[(f^{-1}(\cF))_{\FRP} = f^{-1}(\cF_{\FRP}) \quad (\text{resp. }  (f_!\cG)_{\FRP} = f_!(\cG_{\FRP}), \quad 
    (f_*\cG)_{\FRP} = f_*(\cG_{\FRP})).\]
    
    In particular, given a closed immersion $i \colon Z \to X$ whith complement $j \colon U \to X$, then we have a natural short exact sequence on $X_{\FRP}$ \[ \begin{tikzcd}
        0 \arrow[rr] &  & j_!\Lambda_U \arrow[rr] &  & \Lambda_X \arrow[rr] &  & i_*\Lambda_Z \arrow[rr] &  & 0.
    \end{tikzcd} \]
\end{lemma}
\begin{proof}
    By \autoref{adjunction_FRP_and_etale}.\ref{FRP_embeds_in_etale}, the compatibility $(f^{-1}(\cF))_{\FRP} = f^{-1}(\cF_{\FRP})$ is an immediate consequence of \stacksproj{03CB}. If $f$ is étale, then we have an equality of 
    functors $f^{-1} \circ (-)|_{X_{\et}} = (-)|_{Y_{\et}} \circ f^{-1}$, by \stacksproj{04J0}. Passing to left adjoints, 
    see \autoref{adjunction_FRP_and_etale}, gives precisely that $f_!$ commutes with $(-)_{\FRP}$. If $f$ is proper,
    then the compatibility $(f_*\cG)_{\FRP} = f_*(\cG_{\FRP})$ is a consequence of the proper base change theorem, see \cite[Expos\'e XII, Th{\'e}or{\`e}me 5.1(i)]{SGA_4_tome3}.

    The statement after ``In particular'' now follows from the previous statements by applying the exact functor $(-)_{\FRP}$ to the analogous short exact sequence in the étale site.
\end{proof}

\begin{lemma}\label{adjunction_restriction_and_lower_shriek}
    Let $j \colon U \to X$ be an open immersion, and let $\cF^{\bullet}\in D(U_{\FRP}, \Lambda)$ and $\cG^{\bullet} \in D(X_{\FRP}, \Lambda)$. Then there is natural isomorphism \[ Rj_*\cR\HHom(\cF^{\bullet}, \cG^{\bullet}|_U) \to R\HHom(j_!\cF^{\bullet}, \cG^{\bullet}). \]
\end{lemma}
\begin{proof}
    Let $\cG^{\bullet} \to \cI^{\bullet}$ be a quasi-isomorphism, where $\cI^{\bullet}$ is a K-injective complex and each
    $\cI^k$ is an injective object, see \stacksproj{079P}. Since $(-)|_U$ is right adjoint to the exact functor $j_!$, it preserves K-injective complexes, see \stacksproj{08BJ}.
    Given sheaves $\cF, \cJ \in \Sh(U_{\FRP}, W_n(\bF_q))$ with $\cJ$ injective, then $\HHom(\cF, \cJ)$ is $j_*$-acyclic,
    by \cite[Corollary III.2.13]{Milne_Etale_Cohomology}. Therefore the isomorphism from the statement is the composition
    \begin{multline*}
    Rj_*\cR\HHom(\cF^{\bullet}, \cG^{\bullet}|_U) \cong Rj_*\HHom(\cF^{\bullet} ,\cI^{\bullet}|_U)
    = j_*\HHom(\cF^{\bullet} ,\cI^{\bullet}|_U) \\
    \cong \HHom(j_!\cF^{\bullet}, \cI^{\bullet})  \cong R\HHom(j_!\cF^{\bullet}, \cG^{\bullet}).
    \end{multline*}
\end{proof}

Given an $X$-scheme $Y$, we will write $Y_Z \coloneqq Y \times_X Z$.

\begin{lemma}[{Proof of \cite[Proposition 1.1]{Kato_Duality_theories_for_p_primary_etale_coh_II}}]\label{Kato_key_lemma_closed_immersions}
    Assume that $X$ is affine, and let $T \to Z$ be an FRP morphism with $T$ affine. Then the category of commutative squares 
    \[ \begin{tikzcd}
        T \arrow[d] \arrow[r] & Y \arrow[d] \\
        Z \arrow[r, hook]     & X
    \end{tikzcd} \] with $Y \to X$ FRP and $Y$ affine admits an initial element $T_X^{\can} \to X$, and the natural morphism $T \to (T_X^{\can})_Z$ is an isomorphism. Moreover, if the above diagram is cartesian, then $T_X^{\can}=\Spec \widehat{B}$,
    where $B=\cO_Y(Y)$ and $\widehat{B}$ is the completion of $B$ along the ideal cutting out $Z$ in $X$.
\end{lemma}

\begin{proof}
    Write $X=\Spec R$ and denote by $I$ the ideal cutting out $Z$ in $X$.
    Let $m \geq 1$, and denote by $Z_m=\Spec R/I^m$ the $m$'th thickening of $Z$ in $X$. By \cite[Lemma 1.3]{Kato_Duality_theories_for_p_primary_etale_coh_II} the restriction functor $(Z_m)_{\FRP} \to Z_{\FRP}$ is an equivalence of categories. In particular, we obtain a unique inverse system of flat and relatively perfect morphisms $T_m=\Spec S_m \to Z_m$. 
    Set $S \coloneqq \lim S_m$, where the transition maps are given by $S_m\to S_m\otimes_{R/I^m} R/I^{n-1}=S_{m-1}$,
    and $T_X^{\can} \coloneqq \Spec(S)$. By \stacksproj{0912}, $T_X^{\can}$ is flat over $X$ and each morphism $T_m \to (T_X)^{\can} \times_X Z_m$, $m \geq 1$, is an isomorphism. We claim that $T_X^{\can} \to X$ is relatively perfect.
    Assuming the claim, the fact that $T_X^{\can}$ is an initial object follows by construction and the fact that each functor $(Z_m)_{\FRP} \to Z_{\FRP}$ is an equivalence. Moreover, in case $Y=\Spec B$ and $T=\Spec B/IB$, then $S_m=B/I^m B$ by uniqueness, and we get $S=\bighat{B}$.
    
    Thus it remains to prove the claim, i.e., that $R \to S$ is relatively perfect, or equivalently that the composition 
    \[ \begin{tikzcd}
        F_*R \otimes_R \lim S_m \arrow[rr] &  & \lim \:(F_*R \otimes_R S_m) \arrow[rr] &  & \lim F_*S_m
    \end{tikzcd} \] is an isomorphism. The fact that the first arrow is an isomorphism follows again from \stacksproj{0912} and the $F$-finiteness of $R$, so let us focus on the second arrow. Denote by $I^{[p^m]}$ the ideal generated by all $p^m$-powers of elements of $I$. Since all maps in the inverse system $\{ f_m \colon R/I^{[p^m]} \to R/I^m\}$ are surjective
     and the induced maps $\ker(f_{m'}) \to \ker(f_m)$ are zero for $m' \gg m$, we see that the same properties holds for the inverse system of morphisms \[ \left\{F_*R \otimes_R S_m \cong F_*(R/I^{[p^m]}) \otimes_{R/I^m} S_m \to F_*(R/I^m) \otimes_{R/I^m} S_m \cong F_*S_m\right\}_m, \] where we used that $S_m$ is relatively perfect over $R/I^m$ for the last isomorphism. 
     Hence, we obtain an isomorphism of associated inverse limits, which is precisely what we needed.
\end{proof}

\begin{corollary}[{\cite[Lemma 2.2]{Kato_Duality_theories_for_p_primary_etale_coh_II}}]\label{pushforward_from_closed_immersion_fully_faithful}
    The functor $i_* \colon \Sh(Z_{\FRP}, \Lambda) \to \Sh(X_{\FRP}, \Lambda)$ is exact, fully faithful, and the counit $i^{-1} \circ i_* \to \id$ is an isomorphism. The same holds at the level of derived categories.
\end{corollary}
\begin{proof}
    To see exactness, note that it holds by definition that for any FRP morphism $Y \to X$ and $\cF \in \Sh(Z_{\FRP}, \Lambda)$, we have that $(i_*\cF)|_{Y_{\et}} = i_*(\cF|_{Y_{Z,{\et}}})$ in $\Sh(Y_{\et}, \Lambda)$. 
    Given that $i_* \colon \Sh(Y_{Z,{\et}}, \Lambda) \to \Sh(Y_{\et}, \Lambda)$ is exact, the first statement is proven.

    Let us show that for any $\cF \in \Sh(Z_{\FRP}, \Lambda)$, the natural map $i^{-1}i_*\cF \to \cF$ is an isomorphism. We may assume that $X$ is affine. We then see from \autoref{Kato_key_lemma_closed_immersions} that $i^{-1}i_*\cF$ is the sheafification of the 
    presheaf $\cG$ sending an affine scheme $T \in Z_{\FRP}$ to $\Gamma(T_X^{\can}, i_*\cF) = \Gamma(T, \cF)$, and it follows by construction that the natural map of presheaves $\cG \to \cF$ is an isomorphism on $T$. Thus, the induced map at the sheafification is an isomorphism.
    
    This implies formally that $i_* \colon \Sh(Z_{\FRP}, \Lambda) \to \Sh(X_{\FRP}, \Lambda)$ is fully faithful, and the same holds at the level of derived categories since both $i_*$ and $i^{-1}$ are exact functors.
\end{proof}

\begin{definition}\label{defn:FRP-upper-shriek}
    Let $\cF \in \Sh(X_{\FRP}, \Lambda)$. We define \[ \Gamma_Z(\cF) \coloneqq \ker(\cF \to j_*\cF|_U) \in \Sh(X_{\FRP}, \Lambda), \] and \[ i^{\flat}\cF \coloneqq i^{-1}\Gamma_Z(\cF) \in \Sh(Z_{\FRP}, \Lambda). \]
Note that both functors $\Gamma_Z$ and $i^{\flat}$ are left exact. We will denote by 
\[\RGamma_Z \colon D^+(X_{\FRP}, \Lambda) \to D^+(X_{\FRP}, \Lambda)\quad \text{and}\quad i^!= Ri^{\flat} \colon D^+(X_{\FRP}, \Lambda) \to D^+(Z_{\FRP}, \Lambda)\] 
their respective right derived functors.
\end{definition}

\begin{lemma}\label{derived_version_of_iflat}
    For any $\cF^{\bullet} \in D^+(X_{\FRP}, \Lambda)$, there is a natural exact triangle 
    \[ \begin{tikzcd}
        \RGamma_Z(\cF^{\bullet}) \arrow[rr] &  & \cF^{\bullet} \arrow[rr] &  & Rj_*\cF^{\bullet}|_U \arrow[rr, "+1"] & & {}.
    \end{tikzcd} \] 
\end{lemma}
\begin{proof}
    This is standard, and follows from the fact that if $\cI \in \Sh(X_{\FRP}, \Lambda)$ is an injective object, then $\cI \to j_*\cI|_U$ is surjective, as follows by applying $\HHom(-, \cI)$ to the exact sequence in \autoref{etale_functors_and_FRP_functors_agree} and using the adjunction between $j_!$ and $j^*$. 
\end{proof}

One would hope that $i^{\flat}$ gives a right adjoint to the direct image functor $i_*$, but this seems not to hold in general. We will study when this fact and its derived version are actually correct.

\begin{construction}\label{constr:upper-shriek-FRP}
    Let $\cF \in \Sh(X_{\FRP}, \Lambda)$ and $\cG \in \Sh(Z_{\FRP}, \Lambda)$.
    \begin{enumerate}
        \item By definition and adjunction between $i^{-1}$ and $i_*$, there exists a natural morphism \begin{equation*}\label{nat_tranfso_upper_shriek_not_adj}
            \theta_{\cF} \colon \Gamma_Z(\cF) \longrightarrow i_*i^{\flat}\cF,
        \end{equation*}
        this extends to the derived category, i.e., for any $\cF^\bullet\in D^+(X_{\FRP}, \Lambda)$ we have a natural map
        \[\theta_{\cF^\bullet}\colon R\Gamma_Z(\cF^\bullet)\longrightarrow i_*i^!\cF^\bullet.\]
        
        \item For any $\cF^\bullet\in D^+(X_{\FRP}, \Lambda)$, $\cG^\bullet\in D^+(Z_{\FRP}, \Lambda)$  
        there exists a natural transformation \begin{equation*}\label{nat_transfo_upper_shriek_adj}
            \adj_{\cG^{\bullet}, \cF^{\bullet}} \colon \cR\HHom(i_*\cG^\bullet, \cF^\bullet) \to i_*\cR\HHom(\cG^{\bullet}, i^!\cF^{\bullet})
        \end{equation*} defined as the following composition
        \begin{multline*}
        \cR\HHom(i_*\cG^\bullet, \cF^{\bullet})\simeq \cR\HHom(i_*\cG^{\bullet}, R\Gamma_Z(\cF^{\bullet})) 
        \xrightarrow{\theta_{\cF^{\bullet}}} \cR\HHom(i_*\cG^{\bullet}, i_*i^!\cF^{\bullet})\\
        \simeq         i_*\cR\HHom(i^{-1}i_*\cG^{\bullet}, i^!\cF^{\bullet})\simeq i_*\cR\HHom(\cG^{\bullet}, i^!\cF^{\bullet}),
        \end{multline*}
        where the first equivalence is induced by the vanishing $(i_*\cG^{\bullet})|_U=0$ together with the exact triangle from \autoref{derived_version_of_iflat}, the second equivalence is adjunction of $(i^{-1}, i_*)$ , and the third equivalence
        holds by \autoref{pushforward_from_closed_immersion_fully_faithful}.
    \end{enumerate}
\end{construction}

\begin{lemma}\label{adjoint_iff_}
    Let $\cF^{\bullet} \in D^+(X_{\FRP}, \Lambda)$. Then the following statements are equivalent:
    \begin{enumerate}[label=(\roman*)]
        \item\label{itm:theta_G iso} the morphism $\theta_{\cF^{\bullet}}$ is an isomorphism;
        \item\label{itm:adj_works} for any $\cG^{\bullet} \in D^+(Z_{\FRP},\Lambda)$, the morphism 
        \[ \adj_{\cG^{\bullet}, \cF^{\bullet}} \colon \cR\HHom(i_*\cG, \cF) \to i_*\cR\HHom(\cG, i^!\cF) \] 
        is an equivalence.
    \end{enumerate}
\end{lemma}
\begin{proof} 
    The fact that \ref{itm:theta_G iso} implies \ref{itm:adj_works} is immediate. For the converse implication consider the diagram
    \[\begin{tikzcd} 
    i_*\cR\HHom(\Lambda_Z, i^!\cF^{\bullet})\arrow[r] & \cR\HHom(i_*\Lambda_Z, i_*i^!\cF^{\bullet})\arrow[r] 
               &\cR\HHom(\Lambda_X, i_*i^!\cF^{\bullet})\\
    \cR\HHom(i_*\Lambda_Z, \cF^{\bullet})\arrow[r]\arrow[u, "\adj_{\Lambda_Z,\cF^{\bullet}}"] & \cR\HHom(i_*\Lambda_Z, R\Gamma_Z(\cF^{\bullet}))\arrow[r]\arrow[u, "\theta_{\cF^{\bullet}}"] &   \cR\HHom(\Lambda_X, R\Gamma_Z(\cF^{\bullet})).\arrow[u, "\theta_{\cF^{\bullet}}"]
            \end{tikzcd}\]
    The square on the left commutes by definition of the maps involved and the square on the right commutes by functoriality.
    The  horizontal map on the lower left is an equivalence by \autoref{derived_version_of_iflat} and 
    the vanishing  $j^{-1}i_*\Lambda_Z=0$,
    the horizontal map on the lower right is an isomorphism by \autoref{etale_functors_and_FRP_functors_agree} 
    and as $j^{-1}R\Gamma_Z(\cF^{\bullet})=0$, 
    the composition of the top horizontal map is an isomorphism as source and target are isomorphic to
    $i_*i^!\cF$. If \ref{itm:adj_works} holds, then the vertical map on the left is an equivalence and hence
    so is the vertical map on the right, which gives \ref{itm:theta_G iso}.    
\end{proof}

\begin{lemma}\label{qcoh_is_good_for_right_adjoint_non_derived_version}
   Let $\cM^\bullet \in D^+(\QCoh_{W_nX})$, and set $\Lambda = W_n(\bF_q)$. Then $\theta_{\cM^{\bullet}_{\FRP}}$ is an isomorphism. 
\end{lemma}
\begin{proof}
 If each $\cM^j$ is injective, then combining \autoref{pushforwards_of_qcoh_commute_with_FRP} and \autoref{derived_version_of_iflat}, 
 we deduce that each $\cM^j$ is $\Gamma_Z$-acyclic and therefore $i^{\flat}$-acyclic too. Thus, by considering injective resolutions by quasi-coherent $W_n\cO_X$-modules, it suffices to show that 
    $\theta_{\cM}\colon \Gamma_Z(\cM)\to i_*i^{\flat}\cM$ is an isomorphism for $\cM\in \QCoh_{W_nX}$.

   To this end, first note that given a sheaf $\cF \in \Sh(X_{\FRP}, \Lambda)$, it holds by construction that $\theta_{\cF} = \theta_{\Gamma_Z(\cF)}$. By exactness of $(-)_{\FRP}$ and \autoref{pushforwards_of_qcoh_commute_with_FRP}, we have that $(\Gamma_Z(\cM))_{\FRP} = \Gamma_Z(\cM_{\FRP})$, where $\Gamma_Z(\cM) \coloneqq \ker(\cM \to j_*\cM|_U) \in \QCoh_{W_nX}$. 
    Thus, we may assume that $\cM = \Gamma_Z(\cM)$, i.e. $\cM|_U = 0$. We may also assume that $X$ is affine.

    We will construct a sheaf $\cF \in \Sh(Z_{\FRP}, \Lambda)$ such that $i_*\cF \cong \cM_{\FRP}$. Since the morphism $\theta_{i_*\cF}$ is an isomorphism by \autoref{pushforward_from_closed_immersion_fully_faithful}, this will conclude the proof. To do so, consider the affine FRP site on $Z$, namely the site whose objects are FRP morphisms $f \colon Y \to Z$ with $Y$ affine, together with the étale topology, the induced category of sheaves is therefore equivalent to that of FRP sheaves. Consider the presheaf $\cG$ on the affine FRP site of $Z$ defined by sending $(T \to Z)$ to $\Gamma(T_X^{\can}, \cM_{\FRP})$ (see \autoref{Kato_key_lemma_closed_immersions}). We start with showing that this is a sheaf. Therefore, let $T \to Z$ be an FRP morphism and let $U \to T$ be an étale cover, with both $U$ and $T$ affine. Set $M_{T_X^{\can}} \coloneqq \Gamma(T_X^{\can}, \cM_{\FRP})$ and define $M_{U_X^{\can}}$ and $M_{(U \times_T U)_X^{\can}}$ similarly. We also set $M = \Gamma(X, \cM)$, and we shall use the same notations for any quasi-coherent module on $X$. By definition, we need to show that the diagram \[ \begin{tikzcd}
        M_{T_X^{\can}} \arrow[rr] &  & M_{U_X^{\can}} \arrow[rr, shift left] \arrow[rr, shift right] &  & M_{(U \times_T U)_X^{\can}}
    \end{tikzcd} \] is an equalizer. Let $I$ be the ideal cutting out $Z \inc X$. To ease notations, we will still denote by $I$ the image of this ideal to $\cO_{T_X^{\can}}$, and so on, and we denote by $M[I^m]$ the $I^m$-torsion submodule of $M$. For any $m \geq 1$, it follows by flatness of $T_X^{\can} \to X$ that $M_{T_X^{\can}}[I^m] = (M[I^m])_{T_X^{\can}}$ and similarly for $U_X^{\can}$ and $(U \times_T U)_X^{\can}$. By exactness of colimits and the facts that $M = \bigcup_m M[I^m]$ and that the functor $M\mapsto M_{T^{\can}_X}$ commutes with colimits, it is therefore enough to show that each diagram 
    \[ \begin{tikzcd}
        {M[I^m]_{T_X^{\can}}} \arrow[rr] &  & {M[I^m]_{U_X^{\can}}} \arrow[rr, shift left] \arrow[rr, shift right] &  & {M[I^m]_{(U \times_T U)_X^{\can}}}
    \end{tikzcd} \] is an equalizer. Fix $m \geq 1$, set $Z_m = V(I^m) \inc X$, and write $T_m \coloneqq T_X^{\can} \times_X Z_m$ and similarly for $U_m$ and $(U \times_Z U)_m$. Given that $T_1 = T$, $U_1 = U$ and $(U \times_T U)_1 = U \times_T U$ by \autoref{Kato_key_lemma_closed_immersions}, we obtain by \cite[Lemma 1.3]{Kato_Duality_theories_for_p_primary_etale_coh_II} that $U_m \to T_m$ is étale and $(U \times_T U)_m = U_m \times_{T_m} U_m$. The diagram above is therefore an equalizer by usual étale descent for quasi-coherent sheaves. Thus, we have shown that $\cG$ is in fact a sheaf in the affine FRP site of $Z$, and hence gives rise to a sheaf on $Z_{\FRP}$.

    Let us conclude this proof by showing that $i_*\cG = \cM_{\FRP}$. It is enough to verify this when applied to an FRP morphism $f \colon Y \to X$ with $Y$ affine. By construction, we have $\Gamma(Y, \cM_{\FRP}) = M_Y$, while $\Gamma(Y, i_*\cG) = M_{(Y_Z)_X^{\can}}$, so we want to show that the natural map $M_Y \to M_{(Y_Z)_X^{\can}}$ is an isomorphism. Write $Y = \Spec(S)$. 
    By  \autoref{Kato_key_lemma_closed_immersions} we have  $(Y_Z)_X^{\can} = \Spec(\widehat{S})$, where $\widehat{S}$ 
    denotes the $I$-adic completion of $S$. Thus, we have to show that $M_S \to M_S \otimes_S \widehat{S}$ is an isomorphism. 
    Since $S \to \widehat{S}$ becomes an isomorphism modulo $I^m$ for all $m$ by \stacksproj{0912},
    and $M_S = \bigcup_m M_S[I^m]$, the result follows by a colimit argument.
\end{proof}

Now we consider  $\Lambda = W_n(\bF_q)$. 

\begin{definition}\label{defn:C_adj}
We define $\cC_{X, \adj} \inc D^+(X_{\FRP}, W_n(\bF_q))$ to be the full subcategory consisting of complexes $\cF^{\bullet}$ such that 
for any closed immersion $i\colon Z\inj X$ the natural transformation 
$\theta_{\cF^{\bullet}}\colon R\Gamma_Z(\cF^\bullet)\to i_*i^!\cF^{\bullet}$ is an isomorphism.
\end{definition}

\begin{corollary}\label{exact_upper_shriek_triangle}
\begin{enumerate}[label=(\arabic*)]
    \item\label{exact_upper_shriek_triangle1} The category $\cC_{X,\adj}$ is a full triangulated subcategory of the derived category 
    $D^+(X_{\FRP}, W_n(\bF_q))$, and contains all elements of the form $\cM^{\bullet}_{\FRP}$, 
    with $\cM^{\bullet} \in D^+(\QCoh_{W_n X})$.
    \item \label{exact_upper_shriek_triangle2} If $i\colon Z\inj X$ is a closed immersion, then $i^!(\cC_{X,\adj})\subset \cC_{Z,\adj}$.
    \item\label{exact_upper_shriek_triangle3} If $j\colon U\inj X$ is an open immersion, then $j^{-1}(\cC_{X,\adj})\subset \cC_{U,\adj}$.
\end{enumerate}        
\end{corollary}
\begin{proof}
\ref{exact_upper_shriek_triangle1} follows directly from \autoref{qcoh_is_good_for_right_adjoint_non_derived_version}.
\ref{exact_upper_shriek_triangle3} follows from the fact that for any closed immersion $i\colon Z\inj X$
we have 
\[(j^{-1}\theta_{\cF}\colon j^{-1}\Gamma_Z(\cF)\to j^{-1}(i_*i^{-1}\Gamma_Z(\cF)))= 
(\theta_{j^{-1}\cF}\colon \Gamma_{Z\cap U} (j^{-1}\cF))\to i_{U*}i_U^{-1}\Gamma_{Z\cap U}(j^{-1}\cF)),\]
where $i_U\colon Z\cap U \inj U$ is the closed immersion.
For \ref{exact_upper_shriek_triangle2}, let $i_1\colon Z_1\inj Z$ be a closed immersion, we will write $i(Z_1)$ 
when viewing it as a closed subscheme of $X$. Let $\cF^{\bullet}\in \cC_{X,\adj}$. We have to show that
\begin{equation}\label{exact_upper_shriek_triangle4}
  \theta_{i^!\cF^{\bullet}}:  R\Gamma_{Z_1}(i^{-1}R\Gamma_Z(\cF^\bullet))\to i_{1*}i_1^!R\Gamma_{Z_1}(i^{-1}R\Gamma_Z(\cF^{\bullet}))
\end{equation}
is an isomorphism.
Since $\cF^{\bullet}\in \cC_{X,\adj}$ we have an isomorphism
\begin{equation}\label{exact_upper_shriek_triangle5}
R\Gamma_{i(Z_1)}(\cF^{\bullet})\xrightarrow{\simeq} (i\circ i_1)_* (i\circ i_1)^{-1}R\Gamma_{i(Z_1)}(\cF^{\bullet}).
\end{equation}
Again using that $\cF^{\bullet}\in \cC_{X,\adj}$ we can identify the left hand side with
\[R\Gamma_{i(Z_1)}(\cF^{\bullet})=R\Gamma_{i(Z_1)}R\Gamma_Z(\cF^{\bullet})=R\Gamma_{i(Z_1)}i_*i^{-1}R\Gamma_Z(\cF^{\bullet})=
i_*R\Gamma_{Z_1}(i^{-1}R\Gamma_Z(\cF^{\bullet})).\]
Hence using $i^{-1}i_*=\id$ by \autoref{pushforward_from_closed_immersion_fully_faithful},  
the right hand side of \eqref{exact_upper_shriek_triangle5} becomes
\[i_* (i_{1*}i_1^{-1}R\Gamma_{Z_1}(i^{-1}R\Gamma_Z(\cF^{\bullet}))).\]
Thus applying $i^{-1}$ to \eqref{exact_upper_shriek_triangle5} yields \eqref{exact_upper_shriek_triangle4}.
\end{proof}

\section{Frobenius modules, Cartier modules and their solution functors}

Throughout, we fix a Noetherian and $F$-finite $\bF_q$-scheme $X$.

\subsection{First definitions and duality}

Let us first recall the basic definitions of Frobenius/Cartier modules and crystals, following \cite{Bockle_Pink_Cohomological_Theory_of_crystals_over_function_fields, Blickle_Bockle_Cartier_modules_finiteness_results}.

\begin{definition}\label{def Frob module}
An {\em $r$-Frobenius module over $W_n X$} is a pair $(\cM, \tau_{\cM})$ consisting of a quasi-coherent $W_n\cO_X$-module $\cM$ and an $W_n\cO_X$-linear morphism $\tau_{\cM} \colon \cM \to F^r_*\cM$, called the \emph{structural morphism}. We will usually write $\cM$  instead of $(\cM,\tau_{\cM})$, and since $r$ and $n$ are fixed we will simply refer to $\cM$ as a Frobenius module. A morphism $(\cM, \tau_{\cM})\to (\cN,\tau_{\cN})$ between two Frobenius modules is a $W_n\cO_X$-linear map $\theta \colon \cM\to \cN$ such that the square
    \[ \begin{tikzcd}
     \cM \arrow[r, "\theta"] \arrow[d, "\tau_{\cM}"'] & \cN \arrow[d, "\tau_{\cN}"] \\
     F^r_*\cM \arrow[r, "F^r_*\theta"'] & F^r_*\cN
     \end{tikzcd} \] 
commutes. The category of Frobenius modules is denoted by $\QCoh_{W_nX}^{F^r}$. 
 \begin{enumerate}
    \item We say that the Frobenius module $\cM$ is \emph{coherent} if it is coherent as an element in $\QCoh_{W_nX}$. The full subcategory of coherent Frobenius modules over $W_nX$ is denoted $\Coh_{W_nX}^{F^r}$.
    \item We say that the Frobenius module $\cM$ is \emph{ind-coherent} if it is the union of its coherent sub-Frobenius modules. The full subcategory of ind-coherent Frobenius modules is denoted $\IndCoh_{W_nX}^{F^r}$.
    \item We say that $\cM$ is \emph{nilpotent} if $\tau_{\cM}^n = 0$ for some $n \geq 0$, where we make the abuse of notations \[ \tau_{\cM}^n \coloneqq F^{r(n - 1)}_*\tau_{\cM} \circ \dots \circ F^r_*\tau_{\cM} \circ \tau_{\cM}, \] and we say that it is locally nilpotent if it is a union of nilpotent Frobenius modules.
    
    \item We denote by $\Crys_{W_nX}^{F^r}$ (resp. $\IndCrys_{W_nX}^{F^r}$, $\QCrys_{W_nX}^{F^r}$) the Serre quotients of $\Coh_{W_nX}^{F^r}$ (resp. $\IndCoh_{W_nX}^{F^r}$, $\QCoh_{W_nX}^{F^r}$) by locally nilpotent Frobenius modules.

    \item Note that by adjunction, a morphism $\tau_{\cM} \colon \cM \to F^r_*\cM$ corresponds to a morphism $\tau_{\cM}^* \colon F^{r, *}\cM \to \cM$, called the \emph{adjoint structural morphism}. We say that the Frobenius module $\cM$ is \emph{unit} if 
    $\tau_{\cM}^*$ is an isomorphism.
    \end{enumerate}
\end{definition}

\begin{definition}\label{def Cartier module}
   An {\em $r$-Cartier module over $W_n X$} is a pair $(\cM, \kappa_{\cM})$ consisting of a quasi-coherent $W_n\cO_X$-module $\cM$ and an $W_n\cO_X$-linear morphism $\kappa_{\cM} \colon F^r_*\cM \to \cM$, called the \emph{structural morphism}. We will usually write $\cM$ instead of $(\cM,\kappa_{\cM})$, and since $r$ and $n$ are fixed we will simply refer to $\cM$ as a Cartier module. A morphism $(\cM, \kappa_{\cM})\to (\cN,\kappa_{\cN})$ between two Cartier modules is a $W_n\cO_X$-linear map $\theta \colon \cM\to \cN$ such that the square
     \[ \begin{tikzcd}
    F^r_*\cM \arrow[r, "F^r_*\theta"] \arrow[d, "\kappa_{\cM}"'] & F^r_*\cN \arrow[d, "\kappa_{\cN}"] \\
    \cM \arrow[r, "\theta"'] & \cN
    \end{tikzcd} \]
    commutes. The category of Cartier modules is denoted $\QCoh_{W_nX}^{C^r}$. We define the notions of coherent, ind-coherent, nilpotent and locally nilpotent Cartier modules as in \autoref{def Frob module}. We also define the corresponding Serre quotients.

    As in the Frobenius module case, note that by adjunction, a morphism $\kappa_{\cM} \colon F^r_*\cM \to \cM$ corresponds to a morphism $\kappa_{\cM}^{\flat} \colon \cM \to F^{r, \flat}\cM$, called the \emph{adjoint structural morphism}. We say that a Cartier module $\cM$ is \emph{unit} if $\kappa_{\cM}^{\flat}$ is an isomorphism. The full subcategory of unit Cartier modules is denoted by $\QCoh_{W_nX}^{C^r, \unit}$, and that of ind-coherent and unit Cartier modules is denoted $\IndCoh_{W_nX}^{C^r, \unit}$.
\end{definition}

\begin{remark}
    \begin{enumerate}[label=(\arabic*)]    
        \item It follows by construction that we can pushforward both Frobenius modules and Cartier modules (by quasi-compact and quasi-separated morphisms). Using the reformulation of Frobenius (resp. Cartier) modules in terms of their adjoint structural morphisms, we see that we can pullback Frobenius modules by any morphism (resp. apply the functor $f^{\flat}$ for $f$ a finite morphism). We will see this again later more explicitly.
        \item At certain times in this article, we will implicitly use that injective objects in $\IndCoh_{W_nX}^{C^r}$ are also injective in $\Mod(W_n\cO_X)$ (\cite[Proposition 4.2.4]{Baudin_Duality_between_Witt_Cartier_crystals_and_perverse_sheaves}), for example to argue that taking derived pushforwards in the category of ind-coherent Cartier modules and then forgetting the Cartier structure corresponds to the usual derived pushforward of sheaves.
    \end{enumerate}
\end{remark}

Frobenius and Cartier modules are related through an enhancement of Grothendieck duality, as we recall now.

\begin{construction}[{\cite[Construction 5.1.2]{Baudin_Duality_between_Witt_Cartier_crystals_and_perverse_sheaves}}]\label{construction_pairings}
    Let us define the pairings that give rise to Grothendieck duality in this context.
    \begin{enumerate}
		\item Let $\cM \in \Coh_{W_nX}^{F^r}$ and $\cN \in \IndCoh_{W_nX}^{C^r}$. Define the following Cartier module structure on $\HHom_{W_n\cO_X}(\cM, \cN)$: \[ \left(f \colon \cM|_U \ra \cN|_U \right) \longmapsto \left(\kappa_{\cN} \circ F^r_*f \circ \tau_{\cM} \colon \cM|_U \to \cN|_U \right) \] for all opens $U \inc X$. Due to the presence of $F^r_*f$ in the formula, this morphism is indeed $q^{-1}$-linear. We will often drop the symbol $W_n\cO_X$ in $\HHom_{W_n\cO_X}(-, -)$.
		\item Let $\cM \in \Coh_{W_nX}^{C^r}$ and $\cN \in \IndCoh_{W_nX}^{C^r, \unit}$. Define the following Frobenius module structure on $\HHom(\cM, \cN)$: 
		\[ \left(f \colon \cM|_U \to \cN|_U \right) \longmapsto \left((\kappa_{\cN}^{\flat})^{-1} \circ F^{r, \flat}f \circ \kappa_{\cM}^{\flat} \colon \cM|_U \ra \cN|_U \right) \] for all opens $U \inc X$. Due to the presence of $F^{r, \flat}f$ in the formula, this morphism is indeed $q$-linear.
    \end{enumerate}
    We can take derived functor versions of these two constructions. Given a unit dualizing complex $W_n\omega_X^{\bullet}$, its associated Grothedieck duality functor is denoted \[ \bD_G \coloneqq \cR\HHom(-, W_n\omega_X^{\bullet}). \]
\end{construction}

Let us define the analogue of a dualizing complex in our context.

\begin{definition}
    Let $\iota \colon  \QCoh_{W_nX}^{C^r, \unit} \to \QCoh_{W_nX}$ denote the inclusion. We say that $W_n\omega_X^{\bullet} \in D(\QCoh_{W_nX}^{C^r, \unit})$ is a \emph{unit dualizing complex over $W_nX$} (or simply unit dualizing complex) if $R\iota(W_n\omega_X^{\bullet}) \in D(\QCoh_{W_nX})$ is a dualizing complex for the scheme $W_nX$ (see \stacksproj{0A87}). In this case, we say that $R\iota(W_n\omega_X^{\bullet})$ underlies a unit dualizing complex (and by abuse of notation we will simply denote it by $W_n\omega_X^{\bullet}$.
\end{definition}

\begin{example}\label{basic_example_unit_dc}
    Let $X$ be a separated scheme of finite type over a perfect field $k$, with structural morphism $\pi \colon X \to \Spec k$. Then by \cite[Corollary 5.1.16]{Baudin_Duality_between_Witt_Cartier_crystals_and_perverse_sheaves}, the usual dualizing complex $W_n\omega_X^{\bullet} \coloneqq \pi^!W_n\cO_{\Spec k}$ naturally underlies a unit dualizing complex. In fact, an explicit complex of quasi-coherent unit Cartier modules representing $W_n\omega^\bullet$
    is the residual complex $\pi^{\Delta}W_n(k)$, with Cartier operator induced by the trace along $F^r$, see
    \cite[IV, Theorem 3.1  and Theorem 4.2]{Hartshorne_Residues_and_Duality}.
\end{example}

\begin{thm}\label{duality_Frobenius_and_Cartier}
    Let $X$ be a separated, Noetherian and $F$-finite $\bF_q$-scheme with a unit dualizing complex $W_n\omega_X^{\bullet}$. Then $\bD_G$ satisfies $\bD_G \circ \bD_G \cong id$ and factors through crystals. In particular, it induces an equivalence of categories \[ D^b(\Crys_{W_nX}^{F^r})^{op} \longrightarrow D^b(\Crys_{W_nX}^{C^r}). \] 
\end{thm}
\begin{proof}
    See \cite[Theorem 5.1.12]{Baudin_Duality_between_Witt_Cartier_crystals_and_perverse_sheaves}.
\end{proof}

\subsection{Solution functors on Frobenius modules}

Let us start with recalling the Riemann--Hilbert correspondence for Frobenius modules \cite{Emerton_Kisin_Riemann-Hilbert_correspondence, Bockle_Pink_Cohomological_Theory_of_crystals_over_function_fields, Bhatt_Lurie_RH_corr_pos_char}.

\begin{definition}\label{definition_Sol_Frobenius_modules}
    Let $\cM$ be a Frobenius module on $X$, and let $f \colon Y \to X$ be any morphism of schemes.
    As already mentioned, we can endow $f^*\cM$ with a Frobenius structure via the composition 
    \[ \begin{tikzcd}
        f^*\cM \arrow[rr, "f^*\tau_{\cM}"] &  & f^*F^r_*\cM \arrow[rr] &  & F^r_*f^*\cM,
    \end{tikzcd} \] where the second arrow comes from the natural transformation $f^*F^r_* \to F^r_*f^*$.

    In particular, if $\cM_{\et}$ denotes the extension of the quasi-coherent $W_n\cO_X$-module $\cM$ to the étale site of $W_nX$, 
    we can endow it with a Frobenius module structure (which we will still denote by $\tau_{\cM}$). We then define the solution functor
    \begin{equation}\label{eq:FSol-func}
    \begin{tikzcd}[row sep = tiny]
    \QCoh_{W_nX}^{F^r} \arrow[rr] &  & {\Sh(X_{\et}, W_n(\bF_q))}                       \\
    \cM \arrow[rr, maps to]       &  & \Sol_{\et}(\cM) \coloneqq \ker(\tau_{\cM} - 1 \colon \cM_{\et} \to \cM_{\et})
    \end{tikzcd}
    \end{equation}
\end{definition}

\begin{thm}[\cite{Bockle_Pink_Cohomological_Theory_of_crystals_over_function_fields, Bhatt_Lurie_RH_corr_pos_char}]\label{Riemann_Hilbert_Frobenius_modules} 
    The functor \eqref{eq:FSol-func} annihilates nilpotent Frobenius modules, and induces an equivalence of categories \[ \begin{tikzcd}
        \Crys_{W_nX}^{F^r} \arrow[rr, "\cong"] &  & {\Sh_c(X_{\et}, W_n(\bF_q))},
    \end{tikzcd}\] where the right hand side denotes the category of \'etale  constructible $W_n(\bF_q)$-sheaves.
\end{thm}
\begin{proof}
    This is \cite[Theorem 3.2.3]{Baudin_Duality_between_Witt_Cartier_crystals_and_perverse_sheaves} (see also \cite[Theorem 10.3.6]{Bockle_Pink_Cohomological_Theory_of_crystals_over_function_fields} and \cite[Theorems 9.1.6 and 10.2.7]{Bhatt_Lurie_RH_corr_pos_char}).
\end{proof}

Let us define a variant of the above version of the Riemann--Hilbert correspondence with the FRP site.

\begin{construction}
    Let $\cM \in \QCoh_{W_nX}^{F^r}$. Given $f \colon Y \to X$ an FRP morphism, we can proceed as in \autoref{definition_Sol_Frobenius_modules} to endow $\cM_{\FRP}$ from  \autoref{defn:FRP} with a Frobenius module structure
    (still denoted by $\tau_{\cM}$)
    and set \[ \Sol(\cM) \coloneqq \ker(\tau_{\cM} - 1) \in \Sh(X_{\FRP}, W_n(\bF_q)). \]
\end{construction}

\begin{lemma}\label{Sol_on_Frobenius_modules_comes_from_etale_site}
    The diagram 
    \[ \begin{tikzcd}
        \Coh_{W_nX}^{F^r} \arrow[d, "\Sol_{\et}"'] \arrow[rrd, "\Sol"] &  &                           \\
        {\Sh(X_{\et}, W_n(\bF_q))} \arrow[rr, "(-)_{\FRP}"]                                &  & {\Sh(X_{\FRP}, W_n(\bF_q))}
    \end{tikzcd} \] commutes, where $(-)_{\FRP}$ is the functor from  \autoref{adjunction_FRP_and_etale}\ref{FRP_embeds_in_etale}.
\end{lemma}
\begin{proof}
    Let $\cM \in \Coh_{W_nX}^{F^r}$, and let $f \colon Y \to X$ be an FRP morphism. 
    By \autoref{adjunction_FRP_and_etale}\ref{adjunction_FRP_and_etale1}\ref{adjunction_FRP_and_etale1c}, we have to show that the natural map $f^{-1}(\Sol_{\et}(\cM)) \to \Sol_{\et}(f^*\cM)$ is an isomorphism. This is in fact true for any morphism $g$. Indeed, when $g$ is of finite type this follows from \cite[Theorem 3.2.3]{Baudin_Duality_between_Witt_Cartier_crystals_and_perverse_sheaves} (see also \cite[Proposition 10.2.2]{Bockle_Pink_Cohomological_Theory_of_crystals_over_function_fields}). The general case follows by taking colimits.
\end{proof}

\begin{corollary}\label{Sol_Frobenius_modules_fully_faithful_and_exact}
    The functor $\Sol \colon \Crys_{W_nX}^{F^r} \to \Sh(X_{\FRP}, W_n(\bF_q))$ is fully faithful and exact. The functor $\Sol \colon D^b(\Crys_{W_nX}^{F^r}) \to D^b(X_{\FRP}, W_n(\bF_q))$ is also fully faithful.
\end{corollary}
\begin{proof}
    This follows from \autoref{adjunction_FRP_and_etale}\ref{FRP_embeds_in_etale}, \autoref{Sol_on_Frobenius_modules_comes_from_etale_site} and \autoref{Riemann_Hilbert_Frobenius_modules}.
\end{proof}

We will need the following result later on:

\begin{lemma}\label{surjectivity of F - 1}
    Let $\cM \in \IndCoh_{W_nX}^{F^r}$. Then $\tau_\cM - 1$ is surjective on the $\FRP$ site.
\end{lemma}
\begin{proof}
   Since the result is closed under colimits, we may assume that $\cM \in \Coh_{W_nX}^{F^r}$. Since the result is furthermore closed under extensions, filtering $\cM$ by its $\cI$-power torsion submodules, where $\cI=\Ker( W_n\cO_X\to \cO_X)$, we may assume $\cM\in \Coh_X^{F^r}$. Let $f \colon Y \to X$ be a flat and relatively perfect morphism. We have to show that $\tau_{f^*\cM} - 1$ is surjective on $Y_{\et}$. In general, if a given Frobenius module $\cN$ is nilpotent, then $\tau_{\cN} - 1$ is surjective, since its inverse is $-\sum \tau_{\cN}^i$. This also settles the locally nilpotent case. In particular, the result is known for the the kernel and cokernel of the perfection map
    $f^*\cM \to (f^*\cM)^{1/p^{\infty}}$. It is therefore enough to show the result for $(f^*\cM)^{1/p^{\infty}}$. Since this Frobenius module is holomonic in the sense of \cite[Definition 10.1.7]{Bhatt_Lurie_RH_corr_pos_char}, the result follows from \cite[Remark 10.2.6]{Bhatt_Lurie_RH_corr_pos_char}.
\end{proof}

\subsection{Solution functors on Cartier modules}

Our objective here is to define the solution functor for Cartier crystals, and show that it is exact and faithful on both the étale and FRP sites (along with certain compatibilities). However, we will show that it is not full on the étale site, contrary to the situation for Frobenius module. The proof of fullness in the FRP site will be in \autoref{section_main_thm}.

\begin{construction}
    Let $\cM \in \QCoh_{W_nX}^{C^r}$. Given a flat and relatively perfect morphism $f \colon Y \to X$, we can pullback the Cartier module structure on $\cM$. Indeed, given that
    \[ \begin{tikzcd}
    W_nY \arrow[d, "f"'] \arrow[r, "F"] & W_nY \arrow[d, "f"] \\
    W_nX \arrow[r, "F"] & W_nX \end{tikzcd} \] is a pullback square (\autoref{basic_remark}), we know by flat base change (\stacksproj{02KH}) that the natural map \[ f^*F_*(\cM) \to F_*f^*(\cM) \] is an isomorphism. Therefore, the composition \[ F_*(f^*\cM) \cong f^*F_*\cM \xrightarrow{f^*\kappa_{\cM}} f^*\cM \] naturally gives $f^*\cM$ the structure of a Cartier module on $Y$.

    This way, we can endow $\cM_{\FRP}$ from \autoref{defn:FRP} with the structure of a Cartier module, 
    which we will still denote by $\kappa_{\cM}$. The restriction of $\cM_{\FRP}$ to the étale site of $X$ is denoted $\cM_{\et}$. 
\end{construction}

\begin{definition}\label{defn:SolC}
    Let $\cM \in \QCoh_{W_nX}^{C^r}$. We define \[ \Sol(\cM) \coloneqq \ker(\kappa_\cM - 1 \colon \cM_{\FRP} \to \cM_{\FRP}) \in \Sh(X_{\FRP}, W_n(\bF_q)). \]
    The restriction of $\Sol(\cM)$ to the étale site of $X$ (resp. the Zariski site) is denoted $\Sol(\cM)_{\et}$ (resp. $\Sol(\cM)_{\Zar})$.
\end{definition}

\begin{proposition}\label{surjectivity of C - 1 easier}
    Let $\cM \in \IndCoh_{W_nX}^{C^r}$. Then \[ \kappa_{\cM} - 1 \colon \cM_{\FRP} \to \cM_{\FRP} \] is surjective on $X_{\FRP}$.
\end{proposition}

\begin{proof}
    First of all, note that the statement is closed under extensions and filtered colimits, 
    see \autoref{adjunction_FRP_and_etale} and \autoref{defn:FRP}, so we may assume that $\cM \in \Coh_{X}^{C^r}$. Note also that whenever $\cM$ is nilpotent, $\kappa_\cM - 1$ is automatically bijective, since $-\sum_i\kappa_{\cM}^i$ is its inverse. 
    
    Let $\cM'$ denote the image of $\kappa_{\cM}^e \colon F^e_*\cM \to \cM$ for $e \gg 0$ (this stabilizes by \cite[Lemma 13.1]{Gabber_notes_on_some_t_structures}). Then the quotient $\cM/\cM'$ is nilpotent, so by the previous paragraph it is enough to show the result for $\cM'$, whose structural morphism is surjective by construction. 
    
    The statement is local, so we may also assume that $X$ is affine. By \cite[Remark 13.6]{Gabber_notes_on_some_t_structures}, there exists a closed immersion $i \colon X \hookrightarrow Y$ with $Y$ regular and $F$-finite. Combining \autoref{pushforward_from_closed_immersion_fully_faithful} and \autoref{pushforwards_of_qcoh_commute_with_FRP}, it is therefore enough to show the result for $\cN \coloneqq i_*\cM'$. 
    Write $Y = \Spec R$, fix $y \in Y$ and let $N = \Gamma(\Spec R, \cN)$. Throughout the proof, we will allow ourselves to restrict the situation Zariski locally around $y$.
    
    Let $P \surj N$ be a surjection from a finite free $R$-module which is an isomorphism at the residue field at $y$ (which we denote $k(y)$). We therefore have a diagram of modules
    \[ \begin{tikzcd}
        F^r_*P \arrow[d, two heads]   &  & P \arrow[d, two heads] \\
        F^r_*N \arrow[rr, "\kappa_N"] &  & N. \end{tikzcd} \]
    Since $R$ is regular and we work Zariski locally around $y$, we may assume that the module $F^r_*R$ is free by \cite[Theorem 2.1]{Kunz_Characterizations_of_regular_local_rings_of_char_p}, so we can complete this diagram as follows:
    \[ \begin{tikzcd}
        F^r_*P \arrow[d, two heads] \arrow[rr, "\kappa_P"] &  & P \arrow[d, two heads] \\
        F^r_*N \arrow[rr, "\kappa_N"]                      &  & N.   
    \end{tikzcd} \]
    Since $P \to N$ is surjective, so is the induced morphism of FRP sheaves $\ttilde{P}_{\FRP} \to \ttilde{N}_{\FRP}$. Hence, it is enough to show that $\kappa_{\ttilde{P}} - 1 \colon \ttilde{P}_{\FRP} \to \ttilde{P}_{\FRP}$ is surjective. Given that $\kappa_N \colon F^r_*N \to N$ is surjective and $P \to N$ is an isomorphism at $k(y)$, we deduce that $\kappa_P \colon F^r_*P \to P$ must be surjective at $k(y)$ too. By Nakayama's lemma, $\kappa_P$ is therefore surjective around $y$. Since $P$ is free, we then obtain the existence of a splitting $\tau_P \colon P \to F^r_*P$ of $\kappa_P$. We then know by \autoref{surjectivity of F - 1} that $\tau_{\ttilde{P}} - 1$ is surjective on the FRP site. Since $\kappa_{\ttilde{P}} \colon \ttilde{P}_{\FRP} \to \ttilde{P}_{\FRP}$ is surjective, we deduce that the composition $1 - \kappa_{\ttilde{P}} = \kappa_{\ttilde{P}} \circ (\tau_{\ttilde{P}} - 1)$ is also surjective on the FRP site.
\end{proof}

\begin{corollary}\label{Sol_exact_and_commutes_with_pushforwards}
    The functor $\Sol$ is exact and commutes with filtered colimits on ind-coherent Cartier modules. Moreover, for any separated morphism $f \colon X \to Y$ of finite type between Noetherian and $F$-finite schemes, the following diagram \[ \begin{tikzcd}
        D^+(\IndCoh_{W_nX}^{C^r}) \arrow[rr, "Rf_*"] \arrow[d, "\Sol"] &  & D^+(\IndCoh_{W_nY}^{C^r}) \arrow[d, "\Sol"] \\
        {D^+(X_{\FRP}, W_n(\bF_q))} \arrow[rr, "Rf_*"]                 &  & {D^+(Y_{\FRP}, W_n(\bF_q))}.
    \end{tikzcd} \] commutes. The same statements hold for $\Sol_{\et}$.
\end{corollary}

\begin{proof}
    Thanks to \autoref{surjectivity of C - 1 easier} and the snake lemma, the functor $\Sol$ is exact. It commutes with filtered colimits as so does $(-)_{\FRP}$, see  \autoref{adjunction_FRP_and_etale} and \autoref{defn:FRP}.
     Given that $\Sol$ and $f_*$ commute, it remains to show that if $\cI \in \IndCoh_{W_nX}^{C^r}$ is an injective object, 
     then $\Sol(\cI)$ is $f_*$-acyclic. 
    
    First of all, note that $\cI$ is injective as an object in $\Mod(W_n\cO_X)$ by \cite[Proposition 4.2.4]{Baudin_Duality_between_Witt_Cartier_crystals_and_perverse_sheaves}, so we deduce that $\cI_{\FRP}$ is $f_*$-acyclic by \autoref{pushforwards_of_qcoh_commute_with_FRP}. By definition and \autoref{surjectivity of C - 1 easier}, we have a short exact sequence \[ \begin{tikzcd}
        0 \arrow[rr] &  & \Sol(\cI) \arrow[rr] &  & \cI_{\FRP} \arrow[rr, "\kappa_{\cI} - 1"] &  & \cI_{\FRP} \arrow[rr] &  & 0.
    \end{tikzcd} \]
    We therefore see from the $f_*$-acyclicity of $\cI_{\FRP}$ and the surjectivity of $\kappa_{f_*\cI} - 1 \colon f_*\cI_{\FRP} \to f_*\cI_{\FRP}$ (see \autoref{surjectivity of C - 1 easier}) that also $\Sol(\cI)$ is $f_*$-acyclic. The same proof goes through with $\Sol_{\et}$ instead of $\Sol$.
\end{proof}

Note that the functor $\Sol$ vanishes on locally nilpotent objects. By exactness (see \autoref{Sol_exact_and_commutes_with_pushforwards}) and \stacksproj{02MS}, we obtain a factorization $\Sol \colon \IndCrys_{W_nX}^{C^r} \to \Sh(X_{\FRP}, W_n(\bF_q))$. The same holds for $\Sol_{\et}$.

\begin{proposition}\label{Sol_faithful_on_etale_and_FRP_site}
     The solution functors $\Sol_{\et} \colon \IndCrys_{W_nX}^{C^r} \to \Sh(X_{\et}, W_n(\bF_q))$ and $\Sol \colon \IndCrys_{W_nX}^{C^r} \to \Sh(X_{\FRP}, W_n(\bF_q))$ are faithful.
\end{proposition}
\begin{proof}
    By exactness of these functors (see \autoref{Sol_exact_and_commutes_with_pushforwards}) and \stacksproj{06XK}, we have to show that given $\cM \in \IndCoh_{W_nX}^{C^r}$, if $\Sol(\cM) = 0$ (resp. $\Sol(\cM)_{\et}) = 0$), then $\cM$ is locally nilpotent. Since $\Sol(\cM) = 0$ implies in particular that $\Sol(\cM)_{\et} = 0$, it is enough to show that if $\Sol(\cM)_{\et} = 0$, then $\cM$ is locally nilpotent.

    Assume therefore that $\Sol(\cM)_{\et} = 0$ and that $\cM$ is not locally nilpotent. Again by exactness and a colimit argument using \autoref{Sol_exact_and_commutes_with_pushforwards}, we may assume that $\cM \in \Coh_{X}^{C^r}$. As in the proof of \autoref{surjectivity of C - 1 easier}, we may also assume that $\kappa_{\cM} \colon F^r_*\cM \to \cM$ is surjective. Let $\eta$ be the generic point of an irreducible component of $\Supp(\cM)$, and set $S \coloneqq \cO_{X, \eta}$. By \cite[Remark 13.6]{Gabber_notes_on_some_t_structures}, there exists a regular $F$-finite local ring $R$ together with a surjection $R \surj S$. We may therefore see $M \coloneqq \cM_{\eta}$ as a Cartier module on $R$ supported at the closed point $\eta = \{\fm\}$ of $\Spec R$. Since $\fm^nM = 0$ for some $n \geq 1$ and $\kappa_M^n: F^{nr}_* M\to M$ is surjective, we have $\fm M = 0$.
    
    As in the proof of  \autoref{surjectivity of C - 1 easier}, there exists a finite free $R$-module $P$ together with a Cartier module structure $\kappa_P$ with the following properties: 
    \begin{enumerate}
        \item there exists a surjection $P \surj M$ of Cartier modules, such that the underlying morphism of $R$-modules is an isomorphism at the generic point;
        \item the morphism $\kappa_P \colon F^r_*P \to P$ admits a splitting $\tau_P \colon P \to F^r_*P$ (in particular this gives $P$ the structure of a Frobenius module).
    \end{enumerate}
    
    Let $K$ denote the kernel of $P \to M$. Note that since $P \to M$ is an isomorphism at $\eta$ and $\fm M = 0$, we must have that $K = \fm P$. In particular, $\tau_P \colon P \to F^r_*P$ also induces a Frobenius module structure $\tau_K$ on $K$ splitting $\kappa_K$. The cokernel $M$ then acquires a Frobenius module structure $\tau_M$ splitting $\kappa_M$ too. In particular, any $\tau_M$-fixed point is also a $\kappa_M$-fixed point. We therefore see that $M_{\et}$ admits no non-trivial $\tau_M$-fixed point, so $(M, \tau_M)$ is a nilpotent Frobenius module by the Riemann-Hilbert correspondence \autoref{Riemann_Hilbert_Frobenius_modules}. However, $\tau_M$ is injective since $\kappa_M \circ \tau_M = \id$, so we deduce that $M = 0$.
\end{proof}

So far, everything on the Cartier module side seems to reflect the situation in the Frobenius module side. Let us now point out striking differences:

\begin{proposition}\label{suprising_Zariski_faithfulness}
    Assume that $X$ is of finite type over an algebraically closed field $k$, and let $\cM \in \IndCoh_{W_nX}^{C^r}$. Then $\Sol(\cM)_{\Zar} = 0$ if and only if $\cM \sim_C 0$.
\end{proposition}
\begin{remark}
    The analogous result is definitely wrong for Frobenius modules. For example, assume that $X$ admits a non-trivial finite étale cover $f \colon Y \to X$. Then the Frobenius module associated to the sheaf $f_!(\bF_q)$ under the Riemann--Hilbert correspondence has no fixed point on $X_{\Zar}$, but it is not nilpotent. 

    Here is a concrete instance of the general example above: take $X = \Spec k[t^{\pm 1}]$ and $\tau \colon \cO_X \to F^r_*\cO_X$ defined by $\tau(s) = ts^q$. Then in order to obtain a non-zero fixed point, one must give a $(q - 1)$'th root of $t^{-1}$ (which is impossible to achieve on $X_{\Zar})$.

    One way to interpret this discrepancy is that Cartier modules have a tremendeous amount of fixed points. Another instance of this observation is that if $\cM$ is a Cartier module, then $\Sol(\cM)_{\et}$ is in general \emph{not} constructible, unlike the Frobenius module case. For example, if $X$ is a smooth {\em $k$-curve} and $r=1$ (i.e. $q=p$), 
    then $\Sol(\omega_X)_{\et}\cong \cO^\times_X/(\cO^{\times}_X)^p$, by
    e.g. \cite[0, Corollaire 2.1.18]{Illusie_Complexe_de_de_Rham_Witt_et_cohomologie_cristalline}.
\end{remark}

\begin{proof}
    Assume that there exists $x \in X$ such that $\cM_x \not\sim_C 0$. We will find a Zariski open neighborhood $U$ of $x$ and $m \in H^0(U, \cM)$ such that $\kappa_\cM(m) = m$.

     Let us first show the result for $X = \Spec k$, and set $M = \Gamma(\Spec k, \cM)$. We may assume that $M$ is finitely generated. Let $M' \inc M$ denote the image of $F^{er}_*(M) \to M$ for $e \gg 0$, so that $F^r_*M' \to M'$ is an isomorphism (note that $M' \neq 0$ since $M$ is not nilpotent). We can therefore take the inverse of the structural morphism to obtain a (perfect) Frobenius module $M' \to F^r_*M'$. We then obtain from \cite[Corollary p.143]{Mumford_Abelian_Varieties} that it has in particular a fixed point, so we are done in this case.

    Let us now tackle the general case. By \cite[Lemma 4.4.10]{Baudin_Duality_between_Witt_Cartier_crystals_and_perverse_sheaves}, we know that there exists an open neighborhood $U \inc X$ of $x$ such that $\Gamma(U, \cM) \not\sim_C 0$ (we see it as an ind-coherent Cartier module over $\Spec k$ by \cite[Proposition 4.4.4]{Baudin_Duality_between_Witt_Cartier_crystals_and_perverse_sheaves}). We then know by the previous case that the Cartier module $\Gamma(U, \cM)$ has a non-zero fixed point, so the proof is complete.
\end{proof}

\begin{remark}\label{RH_not_okay_with_etake_site}
    Here are two other fundamental differences with the case of Frobenius modules:
    \begin{enumerate}
        \item Even though we have seen that the functor $\Sol_{\et} \colon \Crys_X^C \to \Sh(X_{\et}, \bF_p)$ is faithful in \autoref{Sol_faithful_on_etale_and_FRP_site}, this functor is not full in general. For example, let $K$ denote the separable closure of $\bF_p(x)$, and consider $\omega_K$ together with its canonical Cartier module structure. Then on the one hand, \[ \Hom_{\Crys_X^C}(\omega_K, \omega_K) = \Hom_{\Sh(X_{\et}, \bF_p)}(\bF_p, \bF_p) = \bF_p, \]  (we used \autoref{Riemann_Hilbert_Frobenius_modules} and \autoref{duality_Frobenius_and_Cartier}, even though one can see this equality above more elementarily). On the other hand, $\Hom_{\Sh(X_{\et}, \bF_p)}(\Sol(\omega_K), \Sol(\omega_K))$ is infinite dimensional, since $X_{\et}$ is trivial and $\Sol(\omega_K)$ is an infinite dimensional $\bF_p$-vector space (all elements $d\log(u)$ for $u \in K^{\times} \setminus (K^{\times})^p$ are non-zero). We believe that $\Hom_{\Sh(\bA^1_{\et}, \bF_p)}(\Sol(\omega_{\bA^1}), \Sol(\omega_{\bA^1}))$ is also large (and hence $\Sol_{\et}$ would also not be full on $\bA^1$ in this case), but this seems trickier to show.
        \item Contrary to \autoref{Sol_on_Frobenius_modules_comes_from_etale_site}, the functor $\Sol$ on Cartier modules does not factor through the étale site. This is already wrong in the case of $\bA^1$ with the Cartier module $\omega_{\bA^1}$ 
        with $r=1$, since in that case $\Sol(\omega_{\bA^1})=\cO^\times_{X_{\FRP}}/(\cO^\times_{X_{\FRP}})^p$, see \autoref{rmk:dualizing-dRW} 
        below,
        and for a general FRP scheme $f \colon Y \to \bA^1$, not all elements of $\Gamma(Y, \cO^\times_Y/(\cO^\times_Y)^p)$ 
        come from a unit on some étale open over $\bA^1$.
    \end{enumerate}
\end{remark}

Recall that if $i \colon Z \inj X$ is a closed immersion and $\cN$ is a Cartier module on $Z$, then $i^{\flat}(\cN) = i^{-1}\HHom(W_n\cO_Z, \cN)$ is a Cartier module via \autoref{construction_pairings}, where we give $W_n\cO_Z$ its canonical Frobenius module structure. On Cartier modules, the functor $i^{\flat}$ is right adjoint to $i_*$, as is direct to check. 
This derives to adjunctions
\[i_* : D^+(\QCoh^{C^r}_{W_nZ})\rightleftarrows  D^+(\QCoh^{C^r}_{W_nX}): i^!,\]
\[i_* : D^+(\IndCoh^{C^r}_{W_nZ})\rightleftarrows  D^+(\IndCoh^{C^r}_{W_nX}): i^!,\]
\[i_* : D^+(\IndCrys^{C^r}_{W_nZ})\rightleftarrows  D^+(\IndCrys^{C^r}_{W_nX}): i^!,\]
where $i^!(-)=i^{-1}\cR\HHom(W_n\cO_Z, -)$.

Recall that we also defined the functors $i^{\flat}$ and $i^!$ on the FRP site in \autoref{section_FRP_site}, and that we introduced the triangulated subcategory $\cC_{\adj}$ on $D^+(X_{\FRP}, W_n(\bF_q))$ 
on which these functors are adjoint above \autoref{exact_upper_shriek_triangle}.

\begin{lemma}\label{Sol_of_Cartier_modules_behave_well_with_adjunction}
    Let $\cM^{\bullet} \in D^+(\IndCoh_{W_nX}^{C^r})$. Then $\Sol(\cM^{\bullet}) \in \cC_{\adj}$.
\end{lemma}
\begin{proof}
    It follows from \autoref{surjectivity of C - 1 easier} that there is an exact triangle \[ \begin{tikzcd}
        \Sol(\cM^{\bullet}) \arrow[rr] &  & \cM^{\bullet}_{\FRP} \arrow[rr, "\kappa - 1"] &  & \cM^{\bullet}_{\FRP} \arrow[rr, "+1"] &  & {}
    \end{tikzcd} \] in $D(X_{\FRP}, W_n(\bF_q))$. The result then follows from \autoref{exact_upper_shriek_triangle}.
\end{proof}

\begin{lemma}\label{Sol_commutes_with_upper_shriek}
    Let $\cM^{\bullet} \in D^b(\IndCoh_{W_nX}^{C^r})$, and let $i \colon Z \inj X$ be a closed immersion. Then there is a natural isomorphism $\Sol(i^!\cM^{\bullet}) \to i^!\Sol(\cM^{\bullet})$. 
\end{lemma}
\begin{proof}
    By \autoref{Sol_of_Cartier_modules_behave_well_with_adjunction}, the datum of a natural transformation $\Sol(i^!\cM^{\bullet}) \to i^!\Sol(\cM^{\bullet})$ is equivalent to a natural transformation $i_*\Sol(i^!\cM^{\bullet}) \to \Sol(\cM^{\bullet})$. Since $\Sol$ and $i_*$ commute, by \autoref{Sol_exact_and_commutes_with_pushforwards}, we can therefore apply $\Sol$ to the counit $i_*i^!\cM^{\bullet} \to \cM^{\bullet}$. To show that this natural transformation is an isomorphism, it is enough to show that it is an isomorphism after applying $i_*$ by \autoref{pushforward_from_closed_immersion_fully_faithful}. Moreover, we may work up to nilpotence. One can deduce from \cite[Remark 5.3.6]{Baudin_Duality_between_perverse_sheaves_and_Cartier_crystals} (see also \cite[Theorem 4.1.1]{Blickle_Bockle_Cartier_crystals}) and a dévissage argument that we have a natural exact triangle 
    \[ \begin{tikzcd}
        i_*i^!\cM^{\bullet} \arrow[rr] &  & \cM^{\bullet} \arrow[rr] &  & Rj_*\cM^{\bullet}|_U \arrow[rr, "+1"] &  & {}
    \end{tikzcd} \] in $D^b(\IndCrys_{W_nX}^{C^r})$, where $j \colon U \inj X$ is the open immersion of the complement of $Z$. On the other hand, we know by \autoref{exact_upper_shriek_triangle} and \autoref{derived_version_of_iflat} that there is a natural exact triangle 
    \[ \begin{tikzcd} i_*i^!\Sol(\cM^{\bullet}) \arrow[rr]           &  & \Sol(\cM^{\bullet}) \arrow[rr]                &  & Rj_*\Sol(\cM^{\bullet})|_U \arrow[rr, "+1"]           &  & {}
    \end{tikzcd} \] and it follows from the definitions that the diagram
    \[ \begin{tikzcd}
        \Sol(i_*i^!\cM^{\bullet}) \arrow[rr] \arrow[d] &  & \Sol(\cM^{\bullet}) \arrow[rr] \arrow[d, "="] &  & \Sol(Rj_*\cM^{\bullet}|_U) \arrow[rr, "+1"] \arrow[d] &  & {} \\
        i_*i^!\Sol(\cM^{\bullet}) \arrow[rr]           &  & \Sol(\cM^{\bullet}) \arrow[rr]                &  & Rj_*\Sol(\cM^{\bullet})|_U \arrow[rr, "+1"]           &  & {}
    \end{tikzcd} \] is commutative, where the left arrow is precisely the natural transformation we constructed. Since both the middle and right arrows are isomorphisms (see \autoref{Sol_exact_and_commutes_with_pushforwards}), the left one is too.
\end{proof}

\section{Link with Kato's duality}\label{section_main_thm}

Throughout, we fix a Noetherian and $F$-finite $\bF_q$-scheme $X$.

\subsection{Kato's duality for $W_n(\bF_q)$-sheaves}

\begin{definition}\label{def:D0}
    We define $D_0(X, W_n(\bF_q))$ to be the full triangulated subcategory of $D(X_{\FRP}, W_n(\bF_q))$ generated by coherent $W_n\cO_X$-modules.
\end{definition}

\begin{remark}
    This category is vast and mysterious. Although it contains coherent $W_n\cO_X$-modules, it also contains any sheaf of the form $\Sol(\cM)$, for $\cM$ a coherent Frobenius or Cartier module. In particular, it contains all sheaves of the form $\cF_{\FRP}$ with $\cF$ a constructible étale sheaf by \autoref{Riemann_Hilbert_Frobenius_modules}. It follows from \autoref{pushforwards_of_qcoh_commute_with_FRP} that it is stable under derived proper pushforwards, but this seems to be the only obviously definable functor on this category (apart from restriction to open subsets). It is also a priori unclear whether it endows any $t$-structure.
\end{remark}

\begin{definition}\label{def:Kato-dual}
    Let $W_n\omega_X^{\bullet}$ be a unit dualizing complex over $W_nX$. We set $W_n\omega_{X, \log}^{\bullet} \coloneqq \Sol(W_n\omega_X^{\bullet})$, and define the functor \[ \bD_K \colon D(X_{\FRP}, W_n(\bF_q))^{op} \longrightarrow D(X_{\FRP}, W_n(\bF_q)) \] given by $\bD_K \coloneqq \cR\HHom_{\FRP}(-, W_n\omega_{X, \log}^{\bullet})$.
\end{definition}

\begin{remark}\label{rmk:dualizing-dRW} Let $k$ be a perfect $\bF_q$-field and  $X$ a $k$-scheme of finite type and of pure dimension $d$. 
Denote by $\pi: W_nX\to \Spec W_n k$ the map induced by the structure map of $X$. 
Then we can take $W_n\omega_X^{\bullet}=\pi^! W_n(k)$ as unit dualizing complex sitting in degrees $[-d,0]$.
Assume $r=1$, i.e., $q=p$. In this case $W_n\omega_{X,\log}^\bullet=\Sol(\pi^!W_n(k))$ 
is equivalent to Kato's complex $K_{n,X}$ from 
\cite[(3.1)]{Kato_Duality_theories_for_p_primary_etale_coh_II}\footnote{In fact Kato's complex is defined (up to sign)
as  ${\rm cone}((\pi^{\Delta}W_n k)_{\FRP}\xrightarrow{C-1} (\pi^{\Delta}W_n k)_{\FRP})[-1]$, 
where $\pi^{\Delta}W_n k$ denotes the residual complex representing $\pi^!W_n k$.} 
and $\bD_K$ agrees with Kato's dualizing functor \cite[p. 265]{Kato_Duality_theories_for_p_primary_etale_coh_II}.
We remark furthermore that the restriction to the \'etale site of $W_n\omega^\bullet_{X,\log}$ is 
equivalent to the \'etale cycle complex $\bZ^c_X(0)/p^n$\footnote{For $U$ an \'etale $X$-scheme the degree $-i$ part of 
$\bZ^c_X(0)(U)$ is the free abelian group generated by closed integral subschemes in $X\times \Delta^i$ of dimension $i$.} 
considered in \cite[2.]{Geisser_Duality_via_cycle_complexes}, i.e.,
\[W_n\omega_{X,\log}^\bullet|_{X_{\et}}\simeq \bZ^c_X(0)/p^n,\]
see \cite[Theorem 6.1]{Ren_Bloch_cycle_complex_coherent_dualizing_complexes}.

In case $X$ is additionally smooth, we have $W_n\omega_X^\bullet= W_n\Omega^d_X[d]$, 
by \cite[I, Theorem 4.1]{Ekedahl_Duality_Hodge_Witt}, where $W_n\Omega^d_X$ denotes the top degree of the de Rham--Witt complex of Bloch--Deligne--Illusie \cite{Illusie_Complexe_de_de_Rham_Witt_et_cohomologie_cristalline}. 
If $Y\to X$ is in $X_{\FRP}$,
then with the notation from \autoref{defn:FRP} we have 
$\Gamma(Y, (W_n\Omega^d_X)_{\FRP})= \Gamma(Y,W_n\Omega^d_Y)$, see \cite[(4.1)]{Kato_Duality_theories_for_the_p_primary_etale_cohomology_I}. Denote by $W_n\Omega^d_{X,\log,\FRP}$
the subsheaf of $(W_n\Omega^d_X)_{\FRP}$ which is locally generated by forms
$\dlog[u_1]\cdots \dlog[u_d]$,  $u_i\in \cO^\times_{X_{\FRP}}$.
By \cite[Proposition (3.4)]{Kato_Duality_theories_for_p_primary_etale_coh_II} we have 
\[W_n\omega^\bullet_{X,\log}= W_n\Omega^d_{X,\log,\FRP}[d].\]
\end{remark}

Next, we prove the following version of Kato's duality theorem \cite{Kato_Duality_theories_for_the_p_primary_etale_cohomology_I, Kato_Duality_theories_for_p_primary_etale_coh_II}:

\begin{thm}\label{Kato_duality}
    The natural transformation given by the evaluation $id \to \bD_K \circ \bD_K$ is an equivalence on objects in $D_0(X_{\FRP}, W_n(\bF_q))$.
\end{thm}

Kato proved in \cite{Kato_Duality_theories_for_the_p_primary_etale_cohomology_I} that in the regular case, \autoref{Kato_duality} holds for a specific choice of a unit dualizing complex. We will deduce the general statement from this.

\begin{construction}
    Let $\cL$ be line bundle admitting the structure of a unit Frobenius module and let $\cN$ be a Cartier module. 
    We can endow $\cL \otimes \cN$ with the structure of a Cartier module as follows:
    \[ \begin{tikzcd}
        F^r_*(\cL \otimes \cN) \arrow[rr, "F^r_*((\tau_\cL^*)^{-1} \otimes \: id)"] &  & {F^r_*(F^{r, *}\cL \otimes \cN)} \arrow[rr, "\cong"] &  & \cL \otimes F^r_*\cN \arrow[rr, "id \:\otimes\: \kappa_{\cN}"] &  & \cL \otimes \cN.
    \end{tikzcd} \]
\end{construction}
\begin{remark}
    \begin{enumerate}
        \item  Although the above could be phrased for vector bundles too, we will not since we do not need them.
        \item  One has to be careful though. When $n = 1$, a unit Frobenius module is automatically a vector bundle \cite[Proposition 6.9.3]{Emerton_Kisin_Riemann-Hilbert_correspondence}. This is not true for higher $n \geq 2$, as for example $W_n\cO_X/p$ with its canonical Frobenius module structure makes it unit.
    \end{enumerate}
\end{remark}

\begin{lemma}\label{comparison_unit_dc}
    Assume that $X$ is separated and connected, and let $W_n\omega_{X, 1}^{\bullet}$ and $W_n\omega_{X, 2}^{\bullet}$ be two unit dualizing complexes on $X$. Then there exists a line bundle $\cL$ with the 
    structure of a unit Frobenius module such that $W_n\omega_{X, 2}^{\bullet} \cong W_n\omega_{X, 1}^{\bullet} \otimes \cL[m]$ as Cartier modules for some $m \in \bZ$.
\end{lemma}
\begin{proof}
    By \cite[V, Theorem 3.1]{Hartshorne_Residues_and_Duality} (or \stacksproj{0ATP}), 
    we have an isomorphism $W_n\omega_{X, 2}^{\bullet} \cong W_n\omega_{X, 1}^{\bullet} \otimes \cL[m]$ in $D^b(\QCoh_{W_nX})$,
    for some line bundle $\cL$ on $W_nX$ and some $m \in \bZ$. Let $\bD_G$ be the Grothendieck duality functor (see \autoref{construction_pairings}) associated to $W_n\omega_{X, 1}^{\bullet}$. Since $W_n\omega_{X, 2}^{\bullet}$ is unit, it follows from the constructions that $\bD_G(W_n\omega_{X, 2}^{\bullet})$ is a unit Frobenius module. As there is a canonical isomorphism of coherent $W_n\cO_X$-modules $\bD_G(W_n\omega_{X, 2}^{\bullet}) \cong \cL^{\vee}[-m]$, 
    the sheaf $\cL^{\vee}$ admits the structure of a unit Frobenius module. 
    Hence $\cL=(\cL^\vee)^\vee$ admits a Frobenius module structure as well (the unit part is crucial here). A similar computation as that of \cite[Lemma 4.1.7]{Baudin_Duality_between_perverse_sheaves_and_Cartier_crystals}, which uses separatedness, shows that $\bD_G(\cL^{\vee}[-m]) \cong W_n\omega_{X, 1}^{\bullet} \otimes \cL[m]$ is an isomorphism of Cartier modules. 
    Since $\bD_G(\cL^{\vee}[-m]) \cong \bD_G(\bD_G(W_n\omega_{X, 2}^{\bullet})) \cong W_n\omega_{X, 2}^{\bullet}$, the proof is complete.
\end{proof}

Compare the following with \cite[Proposition 4.1.1]{Katz_P_adic_properties_of_modular_schemes_and_modular_forms} and its proof.

\begin{lemma}\label{sol_of_unit_Frobenius_module}
    Let $\cL$ be a line bundle underlying the structure of a unit Frobenius module. Then étale locally, $\cL$ is isomorphic to $W_n\cO_X$ with its standard Frobenius module structure. 
\end{lemma}
\begin{proof}
    This is a local computation, so we may assume $\cL \cong W_n\cO_X$ as coherent sheaves. The composition 
    \[ \begin{tikzcd}
        W_n\cO_X \arrow[rr, "nat"] &  & {F^{r, *}W_n\cO_X} \arrow[rr, "\tau^*"] &  & W_n\cO_X
    \end{tikzcd} \] must then be given by multiplication by a unit $u \in \Gamma(X, (W_n\cO_X)^{\times})$. 
    By construction, the morphism $W_n\cO_X \to F^r_*W_n\cO_X$ is then given by $f \mapsto uF^r(f)$. 
    
    This Frobenius module is \'etale locally isomorphic to $(W_n\cO_X, F^r)$ if and only if we
    find some \'etale covering $U\to X$ and $s\in \Gamma(U,W_n(\cO_U)^\times)$ such that 
    $F^r(s)s^{-1}=u$.  This follows from the well-known fact that the Lang map
    \[L: (W_n\cO_X)^\times\longrightarrow (W_n\cO_X)^\times,\quad s\mapsto F^r(s)s^{-1},\]
    is surjective in the \'etale topology. Indeed, for $n=1$ this holds as the $(q-1)$st power map on $\bG_m$ is surjective in the 
     \'etale topology. In general, this follows by induction and the usual Artin-Schreier sequence from the exact sequence
    \[\begin{tikzcd}
     0\ar[r] & \cO_X\ar[rr, "a\mapsto 1+V^{n}(a)"] & & (W_{n+1}\cO_X)^\times \ar[r, "R"] & (W_n\cO_X)^\times\ar[r] & 0,   
    \end{tikzcd}\]
    the fact that the Lang map commutes with $R$, the equality $L(1+V^n(a))= 1+V^n(a^q-a)$, and the snake lemma.
\end{proof}

%\begin{lemma}\label{Sol_of_tensor_product}
%    Let $\cM$ be a $W_n$-Cartier module, and $\cL$ a line bundle admitting the structure of a unit $W_n$-Frobenius module. Then $\Sol(\cL \otimes \cM) \cong \Sol(\cL) \otimes \Sol(\cM)$.
%\end{lemma}
%\begin{proof}
%    By construction of an adjunction, if $s_1 \in \Sol(\cL)$, then the image of $F^{r, *}s_1$ via $F^{r, *}\cL \to \cL$ is exactly $s_1$. It then follows immediately from the explicit construction of the Cartier structure on $\cL \otimes \cN$ that the $\Sol(\cL) \otimes \Sol(\cN)$ is mapped to $\Sol(\cL \otimes \cN)$. This morphism is étale locally an isomorphism by \autoref{sol_of_unit_Frobenius_module}.
%\end{proof}

\begin{corollary}\label{Kato_duality_doesnt_depend_on_unit_dc}
    Let $W_n\omega_{X, 1}^{\bullet}$ and $W_n\omega_{X, 2}^{\bullet}$ be two unit dualizing complexes on $X$. Then the conclusion of \autoref{Kato_duality} holds for $(W_n\omega_{X, 1}^{\bullet})_{\log}$ if and only if it does for $(W_n\omega_{X, 2}^{\bullet})_{\log}$.
\end{corollary}
\begin{proof}
    Since the question is local, this is immediate from \autoref{sol_of_unit_Frobenius_module} and \autoref{comparison_unit_dc}.
\end{proof}

\begin{lemma}\label{lem:dual-regular-case}
Assume $X=\Spec R$ is affine, regular and connected, and set $d\coloneqq {\rm rank}_R(\Omega^1_{R/{\bF_p}})$. Then $W_n\Omega^d_X$, the $d$-th degree  in the de Rham--Witt complex of $X$, is a unit dualizing complex on $X$,
with the unit Cartier module structure the one defined by Kato in 
\cite[Lemma 4.1.2]{Kato_Duality_theories_for_the_p_primary_etale_cohomology_I}.
In particular, for any unit dualizing complex $W_n\omega_X^{\bullet}$, there is an invertible $W_n\cO_X$-module $\cL$ with the structure of a unit Frobenius module and an integer $m$, such that 
$W_n\omega^\bullet_X\cong W_n\Omega^d_X\otimes_{W_n\cO_X}\cL[m]$.    
\end{lemma}
\begin{proof}
 Since this is a local question, we may assume that $X = \Spec R$ is affine and connected. 
   As $R$ is regular and $F$-finite it has a finite $p$-basis, say $r_1,\ldots, r_d$, see e.g. \cite[Lemma 2.1.3]{Bhatt-Blickle-Schwede-Tucker-FFiniteSchemesDualizing}. Kato shows in \cite[Lemma 4.1.2]{Kato_Duality_theories_for_the_p_primary_etale_cohomology_I} that $W_n\Omega^d_X$ admits the structure of a Cartier module (note that the de Rham--Witt Frobenius is surjective in this case).
   The map $S=\bF_p[t_1,\ldots, t_d]\to R$, $t_i\mapsto r_i$ is relatively perfect as follows from the definition of a $p$-basis,
hence it is also flat by a result of Gabber, see 
\cite[Proposition 5.2]{Kato_Duality_theories_for_the_p_primary_etale_cohomology_I}.
Let $f:X\to \Spec(S)=T$ be the induced morphism. By the same argument as in 
\cite[I, Proposition 1.14]{Illusie_Complexe_de_de_Rham_Witt_et_cohomologie_cristalline} using \autoref{basic_remark}\ref{basic_remarkb} we have $f^*W_n\Omega^d_{T}=W_n\Omega^d_X$. As $W_n\Omega^d_{T}$ is a finitely presented 
$W_n\cO_T$-module and $f\colon W_nX\to W_n T$ is flat  we get
\[\cR\HHom_{W_n\cO_X}(W_n\Omega^d_X, W_n\Omega^d_X)=f^*\cR\HHom_{W_n\cO_T}(W_n\Omega^d_T, W_n\Omega^d_T)= f^*W_n\cO_T= W_n\cO_X,\]
where the second equality follows from the fact that $W_n\Omega^d_T$ is a dualizing complex by Ekedahl, 
see \autoref{rmk:dualizing-dRW}. Thus $W_n\Omega^d_X$ is a dualizing complex as well, by \cite[V, Proposition 2.1]{Hartshorne_Residues_and_Duality}. Moreover, the identification  $W_n\Omega^d_{T}\cong \pi^!W_n(\bF_p)$, 
where $\pi:T\to \Spec (\bF_p)$ is the structure map, equips $W_n\Omega^d_T$ with the structure of a unit Cartier module, and hence 
also $W_n\Omega^d_X$ inherits the structure of a unit Cartier module, again by \autoref{basic_remark}\ref{basic_remarkb}. Moreover, 
the proof of \cite[Proposition 3.4]{Kato_Duality_theories_for_p_primary_etale_coh_II} shows that Kato's Cartier operator structure
from \cite[Lemma 4.1.2]{Kato_Duality_theories_for_the_p_primary_etale_cohomology_I} on
$W_n\Omega^d_T$ coincides with the Cartier operator structure stemming from the identification $W_n\Omega^d_{T}\cong \pi^!W_n(\bF_p)$.
This proves the first statement. The last statement follows from \autoref{comparison_unit_dc}.
\end{proof}

We can now prove \autoref{Kato_duality} in the regular case.

\begin{corollary}\label{Kato_duality_for_regular}
    Assume that $X$ is regular, and let $W_n\omega_X^{\bullet}$ be a unit dualizing complex. Then \autoref{Kato_duality} holds for $(X, W_n\omega_X^{\bullet})$.
\end{corollary}
\begin{proof}
   Since this is a local question, we may assume that $X = \Spec R$ is affine and connected. By \autoref{lem:dual-regular-case}, $W_n\Omega^d_X$ is a unit dualizing complex with the Cartier module structure defined as  in
    \cite[Lemma 4.1.2]{Kato_Duality_theories_for_the_p_primary_etale_cohomology_I}
   and Kato showed in \cite[Theorem 4.3 (ii)]{Kato_Duality_theories_for_the_p_primary_etale_cohomology_I} that the duality functor associated to $\Sol(W_n\Omega_X^d)$ induces an equivalence on $D_0(X_{\FRP}, W_n(\bF_q))$\footnote{Kato shows this only for $r=1$, but the same proof works also for general $r$} (we remark that in this proof it is essential to work with $X_{\FRP}$ instead of 
   $X_{\et}$). The statement follows from  \autoref{Kato_duality_doesnt_depend_on_unit_dc}.
\end{proof}

Let us finish with dealing with the singular case now. The approach is similar to that of \cite[\S 4]{Kato_Duality_theories_for_p_primary_etale_coh_II}.

\begin{proof}[Proof of \autoref{Kato_duality}]
   Fix $\cF^{\bullet} \in D_0(X_{\FRP}, W_n(\bF_q))$. Our objective is to show that the evaluation map \[ \cF^{\bullet} \to \cR\HHom(\cR\HHom(\cF^{\bullet}, W_n\omega_{X, \log}^{\bullet}), W_n\omega_{X, \log}^{\bullet}) \] is an isomorphism. This is a local statement, so we may assume that $X$ is affine and connected. By \cite[Remark 13.6]{Gabber_notes_on_some_t_structures}, there exists a closed immersion $i \colon X \inj Y$ into an affine regular, Noetherian, $F$-finite and connected scheme $Y$. Fix a unit dualizing complex $W_n\omega_Y^{\bullet}$ on $Y$ (e.g. $W_n\Omega^{d}_Y$, where $d$ is the length of a $p$-basis of $\cO(Y)$,
   see the proof of \autoref{Kato_duality_for_regular}). By \autoref{Kato_duality_doesnt_depend_on_unit_dc} and \autoref{Sol_commutes_with_upper_shriek}, it is enough to show that \[  \cF^{\bullet} \to \cR\HHom(\cR\HHom(\cF^{\bullet}, i^!W_n\omega_{Y, \log}^{\bullet}), i^!W_n\omega_{Y, \log}^{\bullet}) \] is an isomorphism. By \autoref{pushforward_from_closed_immersion_fully_faithful}, we may show this after applying $i_*$. Since we have a commutative diagram 
   \[ \begin{tikzcd}
        &  & {i_*\cR\HHom(\cR\HHom(\cF^{\bullet}, i^!W_n\omega_{Y, \log}^{\bullet}), i^!W_n\omega_{Y, \log}^{\bullet})} \arrow[d, "\cong"] \\
        i_*\cF^{\bullet} \arrow[rru, "i_*(ev)"] \arrow[rrd, "ev"'] &  & {\cR\HHom(i_*\cR\HHom(\cF^{\bullet}, i^!W_n\omega_{Y, \log}^{\bullet}), W_n\omega_{Y, \log}^{\bullet})} \arrow[d, "\cong"]    \\
        &  & {\cR\HHom(\cR\HHom(i_*\cF^{\bullet}, W_n\omega_{Y, \log}^{\bullet}), W_n\omega_{Y, \log}^{\bullet}),}
    \end{tikzcd} \]
    with vertical maps isomorphisms by \autoref{Sol_of_Cartier_modules_behave_well_with_adjunction}, we are done by \autoref{Kato_duality_for_regular}.
\end{proof}

\subsection{The main result}

\begin{thm}[{First half of \autoref{intro_commutativity}}]\label{main_commutativity_thm}
    Let $X$ be a Noetherian and $F$-finite scheme with a unit dualizing complex $W_n\omega_X^{\bullet}$. Then the diagram 
    \[ \xymatrix{D^b(\Crys_{W_nX}^{F^r})^{op}\ar[r]^{\bD_G}\ar[d]^{\Sol} & D^b(\Crys_{W_nX}^{C^r})\ar[d]^{\Sol}\\
    D^b(X_{\FRP},W_n(\bF_q))^{op}\ar[r]^{\bD_{K}} & D^b(X_{\FRP}, W_n(\bF_q))}
    \] commutes.
\end{thm}

\begin{corollary}[First half of {\autoref{main_thm}}]\label{Sol_fully_faithful}
    Let $X$ be a Noetherian and $F$-finite scheme. Then the functor $\Sol \colon \Crys_{W_nX}^{C^r} \to \Sh(X_{\FRP}, W_n(\bF_q))$ is fully faithful.
\end{corollary}
\begin{proof}
    Note that the constructions \[ (U \inc X) \mapsto \Crys_{W_nU}^{C^r}, \: \Sh(U_{\FRP}, W_n(\bF_q)) \] satisfy effective descent with respect to the Zariski topology. The reason this holds for $\Crys_{W_nU}^{C^r}$ is that it is equivalent to $\Coh_{W_nU}^{C^r, \unit}$ via the unitalization functor as in \cite[Corollary 3.4.9]{Baudin_Duality_between_perverse_sheaves_and_Cartier_crystals}, and it is immediate for the latter categories.

    Thanks to this descent result and the fact that $\Sol$ commutes with restriction to open, this is a local statement. We may therefore assume that $X$ is separated and admits a unit dualizing complex by \cite[Corollary 5.1.16]{Baudin_Duality_between_Witt_Cartier_crystals_and_perverse_sheaves} (see also \cite[Theorem 9.1]{Quasi_F_splittings_III}). In this case, it follows from \autoref{main_commutativity_thm}, \autoref{duality_Frobenius_and_Cartier}, \autoref{Sol_Frobenius_modules_fully_faithful_and_exact} and \autoref{Kato_duality} that $\Sol \colon D^b(\Crys_{W_nX}^{C^r}) \to D(X_{\FRP}, W_n(\bF_q))$ is fully faithful. We then deduce by exactness (\autoref{Sol_exact_and_commutes_with_pushforwards}) that $\Sol \colon \Crys_{W_nX}^{C^r} \to \Sh(X_{\FRP}, W_n(\bF_q))$ is also fully faithful.
\end{proof}

Let us now proceed to the proof of \autoref{main_commutativity_thm}.

\begin{lemma}\label{natural_transformation_general}
    Let $\cM \in \Coh_{W_nX}^{F^r}$ and $\cN \in \IndCoh_{W_nX}^{C^r}$. Then there is an inclusion \[ \Sol(\HHom(\cM, \cN)) \inc \HHom(\Sol(\cM), \Sol(\cN)). \]

    As a consequence, given $\cM^{\bullet} \in D^b(\Coh_{W_nX}^{F^r})$ and $\cN^{\bullet} \in D^b(\IndCoh_{W_nX}^{C^r})$, there is a natural transformation \[ \Sol(\cR\HHom(\cM^{\bullet}, \cN^{\bullet})) \to \cR\HHom(\Sol(\cM^{\bullet}), \Sol(\cN^{\bullet})). \]
\end{lemma}
\begin{proof}
    Recall from \autoref{construction_pairings} that the Cartier module structure on $\HHom(\cM, \cN)$ applied to a morphism $f \colon \cM \to \cN$ sends $m \in \cM$ to $\kappa_\cN(f(\tau_{\cM}(m)))$. We therefore see that if $f$ is fixed by $\kappa_{\HHom(\cM,\cN)}$ and $m \in \cM$ is fixed by $\tau_{\cM}$, then $f(m)$ is fixed by $\kappa_{\cN}$. 

    The derived statement follows from the non-derived one by resolving $\cN^{\bullet}$ by injective objects and using the exactness of $\Sol$, see \autoref{Sol_exact_and_commutes_with_pushforwards}.
\end{proof}

\begin{proposition}\label{existence_natural_transformation}
    There is a natural transformation $\Sol \circ \:\bD_G \to \bD_K \circ \Sol$.
\end{proposition}
\begin{proof}
    Apply \autoref{natural_transformation_general} to $\cN^{\bullet} = W_n\omega_X^{\bullet}$.
\end{proof}

Before we proceed to show that this natural transformation is an isomorphism, we prove some preliminary compatibilities.

\medskip

Let $j \colon U \inj X$ be an open immersion. Recall that there is the extension by zero functor $j_! \colon \Sh_c(U_{\et}, W_n(\bF_q)) \to \Sh_c(X_{\et}, W_n(\bF_q))$, which is exact. Through the Riemann--Hilbert correspondence for Frobenius crystals (\autoref{Riemann_Hilbert_Frobenius_modules}), we can then define a functor $j_! \colon \Crys_{W_nU}^{F^r} \to \Crys_{W_nX}^{F^r}$, which is again automatically exact. Note that the functor that extends étale sheaves to FRP sheaves commutes with $j_!$, see 
\autoref{etale_functors_and_FRP_functors_agree}.

\begin{lemma}\label{compatibility_open_immersion}
    Let $j \colon U \inj X$ be an open immersion, let $\cM^{\bullet} \in D^b(\Crys_{W_nU}^{F^r})$ and $\cN^{\bullet} \in D^b(\Crys_{W_nX}^{C^r})$. 
    Then the diagram
    \[ \begin{tikzcd}[column sep=tiny, font=\small]
    {\Sol(\cR\HHom(j_!\cM^{\bullet}, \cN^{\bullet}))} \arrow[r] \arrow[d, "\cong"'] &  {\cR\HHom(\Sol(j_!\cM^{\bullet}), \Sol(\cN^{\bullet})))} \arrow[r, "\cong"] & {\cR\HHom(j_!\Sol(\cM^{\bullet}), \Sol(\cN^{\bullet}))} \arrow[d, "\cong"] \\
    {\Sol(Rj_*\cR\HHom(\cM^{\bullet}, \cN^{\bullet}|_U))} \arrow[r, "\cong"]       & {Rj_*\Sol(\cR\HHom(\cM^{\bullet}, \cN^{\bullet}|_U))} \arrow[r]      & {Rj_*\cR\HHom(\Sol(\cM^{\bullet}), \Sol(\cN^{\bullet}|_U))}  
    \end{tikzcd} \] commutes.
\end{lemma}
\begin{proof}
    First, resolve $\cN^{\bullet}$ by a complex of injective ind-coherent objects $\cI^{\bullet}$. Fix also a quasi-resolution $\Sol(\cI^{\bullet}) \to \cJ^{\bullet}$. It follows from the explicit constructions for all $i, j \in \bZ$, the diagram 
    \[ \begin{tikzcd}[column sep = small]
        {\Sol(\HHom(j_!\cM^i, \cI^j))} \arrow[r] \arrow[d]  & {\HHom(\Sol(j_!\cM^i), \Sol(\cI^j))} \arrow[r, "\cong"]  & {\HHom(j_!\Sol(\cM^i), \Sol(\cI^j))} \arrow[d] \\
        {\Sol(j_*\HHom(\cM^i, \cI^j|_U))} \arrow[r, "="]    & {j_*\Sol(\HHom(\cM^i, \cI^j))} \arrow[r]                & {j_*\HHom(\Sol(\cM^i), \Sol(\cI^j|_U))}
    \end{tikzcd} \] 
    commutes, in fact both compositions send a map $f$ to $f|_U$ restricted to $\Sol(\cM^i) \inc \cM^i_{\FRP}$.
    By definition of the $\HHom$ complex, this gives that the diagram of complex of FRP sheaves
    \[  \begin{tikzcd}[column sep = small]
        {\Sol(\HHom(j_!\cM^{\bullet}, \cI^\bullet))} \arrow[r] \arrow[d]  & {\HHom(\Sol(j_!\cM^\bullet), \Sol(\cI^\bullet))} \arrow[r, "\cong"]  & {\HHom(j_!\Sol(\cM^\bullet), \Sol(\cI^\bullet))} \arrow[d] \\
        {\Sol(j_*\HHom(\cM^\bullet, \cI^\bullet|_U))} \arrow[r, "="]    & {j_*\Sol(\HHom(\cM^\bullet, \cI^\bullet|_U))} \arrow[r]                & {j_*\HHom(\Sol(\cM^\bullet), \Sol(\cI^\bullet|_U))}
    \end{tikzcd} \]  
    commutes. By naturality, the diagrams  \[ \begin{tikzcd}
        {\HHom(\Sol(j_!\cM^\bullet), \Sol(\cI^\bullet))} \arrow[rr] \arrow[d, "\cong"] &  & {\HHom(\Sol(j_!\cM^\bullet), \cJ^{\bullet})} \arrow[d, "\cong"] \\
        {\HHom(j_!\Sol(\cM^\bullet), \Sol(\cI^\bullet))} \arrow[rr]  &  & {\HHom(j_!\Sol(\cM^\bullet), \cJ^{\bullet})}
    \end{tikzcd} \] and
    \[ \begin{tikzcd}
        {\HHom(j_!\Sol(\cM^\bullet), \Sol(\cI^\bullet))} \arrow[rr] \arrow[d] &  & {\HHom(j_!\Sol(\cM^\bullet), \cJ^{\bullet})} \arrow[d] \\
        {j_*\HHom(\Sol(\cM^\bullet), \Sol(\cI^\bullet|_U))} \arrow[rr]                 &  & {j_*\HHom(\Sol(\cM^\bullet), \cJ^{\bullet})}
    \end{tikzcd} \] also commute, where the horizontal arrows are induced by the map $\Sol(\cI^{\bullet}) \to \cJ^{\bullet}$. 
    Altogether we obtain that the diagram \[ \begin{tikzcd}
        {\Sol(\HHom(j_!\cM^{\bullet}, \cI^\bullet))} \arrow[r] \arrow[d]  & {\HHom(\Sol(j_!\cM^\bullet), \cJ^{\bullet})} \arrow[r, "\cong"]  & {\HHom(j_!\Sol(\cM^\bullet), \cJ^{\bullet})} \arrow[d] \\
        {\Sol(j_*\HHom(\cM^\bullet, \cI^\bullet|_U))} \arrow[r, "="]    & {j_*\Sol(\HHom(\cM^\bullet, \cI^\bullet|_U))} \arrow[r]                & {j_*\HHom(\Sol(\cM^\bullet), \cJ^{\bullet}|_U)}
    \end{tikzcd} \] commutes as well, which is precisely what we had to show by construction and exactness of $j_!$ and both $\Sol$ functors, see \autoref{Sol_exact_and_commutes_with_pushforwards} and \autoref{Sol_Frobenius_modules_fully_faithful_and_exact} for the $\Sol$ functors.
\end{proof}

Here is the anologous statement for closed immersions: 

\begin{lemma}\label{compatibility_closed_immersion}
    Let $i \colon Z \inj X$ be a closed immersion, and let $\cM^{\bullet} \in D^b(\Crys_{W_nZ}^{F^r})$ and $\cN^{\bullet} \in D^b(\Crys_{W_nX}^{C^r})$. 
    Then the diagram 
    \[ \begin{tikzcd}[column sep = tiny, font=\small]
    {{\Sol(\cR\HHom(i_*\cM^{\bullet}, \cN^{\bullet}))}} \arrow[r] \arrow[dd, "\cong"'] & {\cR\HHom(\Sol(i_*\cM^{\bullet}), \Sol(\cN^{\bullet})))} \arrow[r, "="] & {\cR\HHom(i_*\Sol(\cM^{\bullet}), \Sol(\cN^{\bullet}))} \arrow[d, "\cong"] \\                                         &                          & {i_*\cR\HHom(\Sol(\cM^{\bullet}), i^!\Sol(\cN^{\bullet}))}                 \\
    {\Sol(i_*\cR\HHom(\cM^{\bullet}, i^!\cN^{\bullet}))} \arrow[r, "="]                     & {i_*\Sol(\cR\HHom(\cM^{\bullet}, i^!\cN^{\bullet}))} \arrow[r]          & {i_*\cR\HHom(\Sol(\cM^{\bullet}), \Sol(i^!\cN^{\bullet}))} \arrow[u, "\cong"']      
\end{tikzcd} \] commutes, where the isomorphism $\Sol(i^!\cN^{\bullet}) \cong i^!\Sol(\cM^{\bullet})$ is that of \autoref{Sol_commutes_with_upper_shriek} and the top vertical map on the right is an isomorphism by \autoref{exact_upper_shriek_triangle}.
\end{lemma}
\begin{proof}
    The proof follows from the same technique as in \autoref{compatibility_open_immersion}, where we additionally use that if $\cI$ is an injective ind-coherent Cartier module on $W_nX$, 
    then so is $i^\flat\cI$ on $W_nZ$.
\end{proof}

We are now ready to prove our main result:

\begin{proof}[Proof of \autoref{main_commutativity_thm}]
We have to show that the natural transformation $\Sol \circ \: \bD_G \to \bD_K \circ \Sol$ from 
    \autoref{existence_natural_transformation} is an isomorphism. The proof will proceed by Noetherian induction, i.e. we may assume that the result holds for all proper closed subschemes $Z \inc X$. By \autoref{compatibility_closed_immersion}, we know in particular that the result holds for all $\cM^{\bullet} \in D^b(\Crys_{W_nX}^{F^r})$ which are in the image of 
    $i_*\colon D^b(\Crys_{W_nZ}^{F^r})\to D^b(\Crys_{W_nX}^{F^r})$, for some strict closed subscheme $i\colon Z\inj X$.
    
    Throughout, we will use the following common trick: if there is an exact triangle of the form 
    \[ \begin{tikzcd}
        \cM_1^{\bullet} \arrow[rr] &  & \cM_2^{\bullet} \arrow[rr] &  & \cM_3^{\bullet} \arrow[rr, "+1"] &  & {}
    \end{tikzcd}  \] such that the natural transformation from \autoref{existence_natural_transformation} is an isomorphism for $\cM_1^{\bullet}$ and $\cM_3^{\bullet}$, then it also holds for $\cM_2^{\bullet}$. An immediate consequence is that it is enough to show the result for $\cM \in \Crys_X^{F^r}$.  Since $\Sol_{\et}(\cM)$ is a constructible sheaf, there exists a non-empty open $U \inc X$ such that $\Sol_{\et}(\cM)|_U$ is an $\bF_q$-local system. 
    
    Consider the exact sequence 
    \[ \begin{tikzcd}
        0 \arrow[rr] &  & j_!(\cM|_U) \arrow[rr] &  & \cM \arrow[rr] &  & \cM' \arrow[rr] &  & 0.
    \end{tikzcd}\] 
    By definition of $j_!(\cM_{|_{U}})$ (see above \autoref{compatibility_open_immersion})  
    and as $\Sol_{\et}$ on Frobenius crystals is exact and preserves pullback and  proper pushforward, e.g. 
    \cite[Theorem 2.2.14]{Baudin_Duality_between_perverse_sheaves_and_Cartier_crystals}, we have  
    \[\Sol_{\et}(\cM')= i_*\Sol_{\et}(\cM)|_{Z}=\Sol_{\et}(i_*i^*\cM).\]
    Thus by the Riemann--Hilbert correspondence (\autoref{Riemann_Hilbert_Frobenius_modules}), there is a Frobenius crystal $\cM_1$ on 
    $Z$ such that $\cM'=i_*\cM_1$. Thus the result is known for $\cM'$ by the Noetherian induction hypothesis and our discussion above. Using \autoref{compatibility_open_immersion}, we deduce that it is enough to show the result for $\cM|_U$. Since $\Sol_{\et}(\cM)|_U$ is étale locally isomorphic $\bF_q^{\oplus n}$, we know by the Riemann--Hilbert correspondence for Frobenius crystals (which in this case is due to \cite{Katz_P_adic_properties_of_modular_schemes_and_modular_forms}) that $\cM$ is étale locally isomorphic to a direct sum of the Frobenius module $\cO$ (with its canonical Frobenius module structure). Since the statement is étale local, we deduce that it is enough to show the result for this very Frobenius module. This is now immediate.
\end{proof}

\subsection{The essential image of the solution functor}\label{sec:essential-image}

In this section we assume that $X$ is a  Noetherian $F$-finite $\bF_q$-scheme.
Recall that  the Krull dimension of $X$ is finite by \cite[\S 42, Lemma 6]{Matsumura_Commutative_Algebra}.
We assume $W_nX$ has a unit dualizing complex 
(e.g. $X$ is of finite type over a perfect field, see \autoref{basic_example_unit_dc})  
and we  denote by $W_n\omega^\bullet_X$ a fixed choice. 

For $m\le n$ we set 
\[W_m\omega^\bullet_X:= i_m^!W_n\omega^\bullet_X,\]
and $W_m\omega^{\bullet}_{X,\log}\coloneqq \Sol(W_m\omega^{\bullet}_X)\in D^b(X_{\FRP}, W_m(\bF_q))$, where $i_m\colon W_mX\inj W_nX$ is the closed immersion. 
If $h\colon Z\inj X$ is a locally closed immersion we set
\begin{align}\label{eq:W_nomegalogZ}
  W_m\omega^\bullet_{Z,\log}:= h^!W_m\omega^\bullet_{X,\log}:= (\bar{h}^!W_m\omega^{\bullet}_{X,\log})|{_Z} 
  \in D^b(Z_{\FRP}, W_n(\bF_q)),   \end{align}
where $\bar{h}: \bar{Z}\inj X$ is the closure of $h$\footnote{The scheme structure of $\bar{Z}$ does not play a role, see \autoref{lem:red}.} and $\bar{h}^!$ is defined in \autoref{defn:FRP-upper-shriek}. 
Note that by \autoref{Sol_commutes_with_upper_shriek} we have $W_m\omega^{\bullet}_{Z,\log}= \Sol(h^!W_m\omega^{\bullet}_X)$
and $h^!W_m\omega^{\bullet}_X$ is a unit dualizing complex over $W_mZ$.
We denote by 
\[\bD_{K,Z}=\cR\HHom_{\FRP}(-, W_n\omega^{\bullet}_{Z,\log}): D(Z_{\FRP}, W_n(\bF_q))^{\rm op}\to D(Z_{\FRP}, W_n(\bF_q))\]
the induced functor, which is dualizing when restricted to the triangulated subcategory $D_0(X, W_n(\bF_q))$ from \autoref{def:D0},
see \autoref{Kato_duality}. We write $\bD_K$ instead of $\bD_{K,X}$.

\begin{lemma}\label{lem:red}
Let $Z$ be  a closed subscheme of $X$ and $\iota\colon Z_{\red}\inj Z$ be the closed embedding of its reduced subscheme.
Then the inverse image functor
\[\begin{tikzcd}
\iota^{-1}\colon D(Z_{\FRP}, W_n(\bF_q))\arrow[r, "\simeq"] & D(Z_{\red, \FRP}, W_n(\bF_q))
\end{tikzcd}
\]
is an equivalence with inverse the pushforward $\iota_*$. Moreover, we have
\[\iota^{-1}W_n\omega^\bullet_{Z,\log}=W_n\omega^\bullet_{Z_{\red},\log}.\]    
\end{lemma}
\begin{proof}
The first statement follows from the fact that for any FRP scheme $Y\to Z$ the category of \'etale $Y$-schemes is equivalent to the category of \'etale $Y_{\red}$-schemes. The second statement follows from
\[\iota_*W_n\omega^{\bullet}_{Z_{\red},\log}=\iota_*\iota^!W_n\omega^{\bullet}_{Z,\log}= 
R\Gamma_{Z_{\red}}(W_n\omega^{\bullet}_{Z,\log})= W_n\omega^{\bullet}_{Z,\log},\]
where the first equality holds by \eqref{eq:W_nomegalogZ}, the second equality follows from the definition of $i^!$ and 
\autoref{exact_upper_shriek_triangle}, and the last equality is immediate.
\end{proof}

\begin{lemma}\label{lem:reg-dualizing-log}
Assume $X$ is Cohen-Macaulay (CM). Then $W_n\omega^\bullet_{X,\log}$
is a single sheaf on $X_{\FRP}$ sitting in some degree, which we denote by $W_n\omega_{X,\log}$.
Moreover, if $h\colon Z\inj X$ is a locally closed immersion with $Z$ irreducible and CM of codimension $c$,
then $h^!W_n\omega_{X,\log}[c]= W_n\omega_{Z,\log}$.
\end{lemma}
\begin{proof}
If $X$ is CM, then so is $W_nX$, as follows from the exact sequence of abelian sheaves
\[0\to \cO_X\to W_n\cO_X\to W_{n-1}\cO_X\to 0\]
and the characterization of CM by local cohomology \cite[IV, Proposition 2.6]{Hartshorne_Residues_and_Duality}.
Thus $W_n\omega^\bullet_X$ is a sheaf sitting in some degree and by \autoref{surjectivity of C - 1 easier},
so is $W_n\omega^\bullet_{X,\log}$. For the second statement we only have to show the vanishing 
$\cH^i(h^!W_n\omega_{X,\log})$, for all $i\neq c$.  By \autoref{surjectivity of C - 1 easier}
it suffices to show $\cH^i(h^!W_n\omega_X)=0$, for all $i\neq c$. 
But as $W_n\omega_X$ is a dualizing complex with associated dimension function $X\ni x\mapsto d(x)=\codim(\overline{\{x\}}, X)$ 
the complex $h^!W_n\omega_X$ is a dualizing complex with associated dimension function $Z\ni z\mapsto \codim(\overline{\{z\}}, Z)+c$,
by \cite[(3.1.26)]{Conrad_Grothendieck_duality}, which yields the statement.
\end{proof}

\begin{lemma}\label{lem:logses}
Assume $X$ is regular. 
Then there is a short exact sequence on $X_{\FRP}$ for all $m\le n$
\begin{equation}\label{lem:logses1}
        \begin{tikzcd}
        0\arrow[r] & W_{n-m}\omega_{X,\log}\arrow[r, "\pi^m"] & W_{n}\omega_{X,\log}\arrow[r, "\phi^{n-m}"] & W_m\omega_{X,\log}\arrow[r] & 0,
    \end{tikzcd}
    \end{equation}
where $\pi^m$ is the map "lift under $\phi^m$ and multiply by $p^m$".
Moreover, for $r=1=m$ and $W_n\omega_{X,\log}=W_n\Omega^{\dim X}_{X,\log}$ this coincides with the sequence from \cite[(4.1.8)]{Kato_Duality_theories_for_the_p_primary_etale_cohomology_I}.
\end{lemma}
\begin{proof}
Throughout the proof, $F\colon W_{n+1}\Omega_X^*\to W_n\Omega_X^*$ denotes the Frobenius stemming from the 
structure of the de Rham-Witt complex. In particular, for this proof only, in degree zero $F$ becomes a map $W_n X\to W_{n+1}X$.
Consider the short exact sequence of sheaves of $r$-Frobenius modules over $W_nX$
\[\begin{tikzcd}
0\arrow[r] & F^{n-m}_*W_m\cO_X\arrow[r, "V^{n-m}"]  & W_n\cO_X\arrow[r, "R^{m}"] & i_{n-m*}W_{n-m}\cO_X\arrow[r] & 0,
\end{tikzcd}
\]
where $i_{n-m} \colon W_{n-m}X\inj W_nX$ 
is the canonical closed immersion. Applying $\bD_G$ and using that $W_n\omega_X^\bullet$ is a single sheaf (see the proof of 
\autoref{lem:reg-dualizing-log}), we obtain a short exact sequence of unit $r$-Cartier modules
\begin{equation}\label{lem:logses2}
\begin{tikzcd}
    0\arrow[r] & i_{n-m*}W_{n-m}\omega_X\arrow[r, "\underline{p}^m"] & W_n\omega_X\arrow[r, "F^{n-m}"] & 
    F^{n-m}_*W_m\omega_{X}\arrow[r] & 0,
\end{tikzcd}
\end{equation}
where $\underline{p}^m:=\bD_G(R^m)$ and $F^{n-m}:=\bD_G(V^{n-m})$. 
As  $\Sol$ is exact by \autoref{surjectivity of C - 1 easier} and commutes with  pushforward by \autoref{Sol_exact_and_commutes_with_pushforwards}, applying $\Sol$ yields an exact sequence as in \eqref{lem:logses1}
with $\varphi^{n-m}:=\Sol(F^{n-m})$  and $\Sol(\underline{p}^m)$ instead of $\pi^m$.
It remains to show that the map $\pi^m=$ "lift under $\varphi^{m}$ and multiply by $p^m$" is well-defined and maps isomorphically onto
the kernel of $\varphi^{n-m}$.

The statement is \'etale local, so by \autoref{lem:dual-regular-case} we may assume that $W_n\omega_{X}=W_n\Omega^d_X$,
where we use \autoref{sol_of_unit_Frobenius_module} to trivialize \'etale locally a possible unit Frobenius module. 
Furthermore,  as in the proof of \autoref{lem:dual-regular-case} we can reduce to   $X=\Spec \bF_q[t_1,\ldots, t_d]$.
The Cartier module structure of $W_n\Omega^d_X$ 
is in this case induced by Kato's Cartier operator \cite[Lemma 4.1.2]{Kato_Duality_theories_for_the_p_primary_etale_cohomology_I},
which  is uniquely determined  by the equality
\begin{equation}\label{lem:logses3}
C(F(\alpha))= R(\alpha),\quad \alpha\in W_{n+1}\Omega^d_X,    
\end{equation}
as $F\colon W_{n+1}\Omega^d_X\to W_n\Omega^d_X$ is surjective. 

If $r=1=m$, then \autoref{lem:logses1} together with the explicit description of $\pi^m$ is stated in 
\cite[Lemma 4.1.5]{Kato_Duality_theories_for_the_p_primary_etale_cohomology_I}. We give the argument for general $r$, $m$.
\footnote{In the case $r=1$ we have $F=R$ on $W_n\omega_{X,\log}$, for general $r$ this is not clear to us.} Note that in this case the map $F^{n-m}=\bD_G(V^{n-m})$
is induced by the $(n-m)$-fold iterated de Rham--Witt Frobenius as follows from \cite[III, Proposition 2.4]{Ekedahl_Duality_Hodge_Witt} (or \cite[Theorem 1.10.1]{Rulling_Chatzimatiaou_Hodge_Witt_cohomology_and_Witt_rational_singularities}).
As the de Rham--Witt Frobenius commutes with the Cartier operator 
we obtain an induced map $\varphi^{n-m}=\Sol(F^{n-m})\colon W_n\omega_{X,\log}\to W_m\omega_{X,\log}$,
which is surjective, as follows from the surjectivity of $F^{n-m}$ on the top forms and the surjectivity of $C^r-1$, 
see \autoref{surjectivity of C - 1 easier}. 

We show first that there is a well-defined  and injective map 
\begin{equation}\label{lem:logses4}
\pi^m=\text{"lift under $\phi^m$ and multiply by $p^m$"}\colon W_{n-m}\omega_{X,\log}\to W_n\omega_{X,\log}.
\end{equation}
To this end note that all the structure results for $W_n\Omega^\bullet_X$ from 
\cite[I, Section 3]{Illusie_Complexe_de_de_Rham_Witt_et_cohomologie_cristalline} can be extended to $W_n\Omega^\bullet_{X,\FRP}$
using \autoref{basic_remark}, see \cite[4.1]{Kato_Duality_theories_for_the_p_primary_etale_cohomology_I}. 
Thus if $\beta\in W_n\omega_{X,\log}\subset W_n\Omega^d_{X,\FRP}$ is a local section with $F^{m}(\beta)=0$,
then $\beta= V^{n-m}(\gamma)$, for some $\gamma\in W_{m}\Omega^d$, by \cite[I, (3.21.1.2)]{Illusie_Complexe_de_de_Rham_Witt_et_cohomologie_cristalline}, and hence $p^m \beta=0$.
This show that \autoref{lem:logses4} is well-defined.
For injectivity, let $\alpha_0=F^m(\alpha)\in W_{n-m}\Omega^{d}_{X,\FRP}$ 
with  $\alpha\in W_n\Omega^d_{X,\FRP}$ satisfying $C^r(\alpha)=\alpha$ and $p^m\alpha=0$, i.e., $\pi^m(\alpha_0)=0$.
By \cite[I, Proposition 3.4]{Illusie_Complexe_de_de_Rham_Witt_et_cohomologie_cristalline} we find $\beta\in W_{m}\Omega^d_{X,\FRP}$
and $\gamma\in W_n\Omega^{d-1}_{X,\FRP}$ such that $\alpha=V^{n-m}(\beta)+dV^{n-m}(\gamma)$. 
As 
\[CdV^{n-m}(\gamma)=CFdV^{n-m+1}(\gamma)=dV^{n-m+1}(R(\gamma)),\]
we see that $C^{rN}(dV^{n-m}(\gamma))=0$, for $N$ big enough. Hence we may assume $\alpha=V^{n-m}(\beta)$, in which case
we have $\alpha_0=F^m(\alpha)=0$. Thus $\pi^m$ is injective.

It remains to show that $\pi^m$ surjects onto the kernel of $\varphi^{n-m}$.
To this end let $\alpha\in W_n\Omega^d_{X,\FRP}$ be a local section such that $F^{n-m}(\alpha)=0$ and $C^r(\alpha)=\alpha$.
By \cite[I, (3.21.1.2)]{Illusie_Complexe_de_de_Rham_Witt_et_cohomologie_cristalline}
there exist $\beta\in W_{n-m}\Omega^d_{X,\FRP}$ such that
$\alpha=V^m(\beta)$. Choose $N$ such that $Nr\ge m$. 
Then we get $V^m(\beta)= \alpha= C^{rN}(\alpha)=V^m(C^{rN}(\beta))$ and applying $\underline{p}^{Nr}$ yields
\[V^{m+Nr}(F^{Nr}(\tilde{\beta}))=V^{m+Nr}(\beta),\]
where $\tilde{\beta}\in W_{n-m+Nr}\Omega^d_{X,\FRP}$ is a lift of $\beta$. 
Hence by \cite[I, (3.21.1.4)]{Illusie_Complexe_de_de_Rham_Witt_et_cohomologie_cristalline}
there exists  $\epsilon\in W_{m+Nr}\Omega^{d-1}_{X,\FRP}$ such that
\[F^{Nr}(\tilde{\beta})-\beta= F^{m+Nr}dV^{n-m}(\epsilon).\]
Thus 
\[\alpha=V^m(\beta)= p^m(F^{Nr-m}(\tilde{\beta})- F^{Nr}dV^{n-m}(\epsilon))=p^m\epsilon_1,\]
for some $\epsilon_1\in W_n\Omega^d_{X,\FRP}$.
By the surjectivity of $\varphi^m$ we find $\epsilon_2\in W_n\omega_{X,\log}$ with $\varphi^m(\epsilon_2)=F^m(\epsilon_1)$.
By \cite[I,(3.21.1.2)]{Illusie_Complexe_de_de_Rham_Witt_et_cohomologie_cristalline} we find an $x\in W_m\Omega^d_{X,\FRP}$, 
such that $\epsilon_1= \epsilon_2+ V^{n-m}(x)$. Hence
\[\alpha= p^m \epsilon_1= p^m \epsilon_2 =\pi^m(\varphi^m(\epsilon_2)),\]
 which yields the exactness in the middle of \eqref{lem:logses1}.
\end{proof}

\begin{definition}\label{defn:dual-lisse}
Assume $X$ is regular.  A {\em dual-lisse sheaf} of $W_n(\bF_q)$-modules on $X_{\FRP}$ 
is a sheaf  isomorphic to 
\[ E(W_n\omega_{X,\log}):=E_{\FRP}\otimes_{W_n(\bF_q)} W_n\omega_{X,\log},\]
where $E$ is a lisse sheaf of $W_n(\bF_q)$-modules on $X_{\et}$ (i.e. locally constant) 
and $E_{\FRP}$ is its extension to $X_{\FRP}$,  see \autoref{adjunction_FRP_and_etale}.
\end{definition}

\begin{lemma}\label{lem:dual-lisse-on-regular}
Assume $X$ is regular. Let $E$ be a lisse sheaf of $W_n(\bF_q)$-modules on $X_{\et}$.
\begin{enumerate}[label=(\arabic*)]
    \item \label{lem:dual-lisse-on-regular2}
    $E_{\FRP}\otimes^L_{W_n(\bF_q)} W_n\omega_{X,\log}= E(W_n\omega_{X,\log})$;
    \item \label{lem:dual-lisse-on-regular3}
    $\cR\HHom(E_{\FRP}, W_n\omega_{X,\log})= E^{\vee}(W_n\omega_{X,\log})$,
    where $E^\vee=\HHom_{X_{\et}}(E, W_n(\bF_q))$;
    \item \label{lem:dual-lisse-on-regular4}
    $h^!(E(W_n\omega_{X,\log}))[c]= E|{_Z}(W_n\omega_{Z,\log})$,
    where $Z$ is regular and connected and $h\colon Z\inj X$ is a locally closed immersion of codimension $c$.
\end{enumerate}
In particular, all the complexes appearing in \ref{lem:dual-lisse-on-regular2}--\ref{lem:dual-lisse-on-regular4} are 
isomorphic to a dual-lisse sheaf.
\end{lemma}

\begin{proof}
Let us prove \ref{lem:dual-lisse-on-regular2}. We have to show the vanishing ${\rm Tor}_i^{W_n(\bF_q)}(E_{\FRP}, W_n\omega_{X,\log})=0$, for $i\ge 0$.
This statement is \'etale local on $X$  hence we may assume, that $E=W_m(\bF_q)$, for some $m\le n$.
\autoref{lem:logses} yields a resolution 
\[\ldots \to W_n\omega_{X,\log}\xrightarrow{p^{n-m}} W_n\omega_{X,\log}\xrightarrow{p^m} 
W_n\omega_{X,\log }\xrightarrow{\varphi^{n-m}} W_m\omega_{X,\log}\to 0.\]
Here the explicit description of $\pi^m$ in \autoref{lem:logses} is important.
Since the free resolution
\begin{equation}\label{lem:dual-lisse-on-regular20}
    \ldots\to W_n(\bF_q)\xrightarrow{p^{n-m}}W_n(\bF_q)\xrightarrow{p^m} W_n(\bF_q)\to W_m(\bF_q)\to 0
\end{equation}
induces  a free resolution of the constant $W_n(\bF_q)_{\FRP}$-module $W_m(\bF_q)_{\FRP}$ on $X_{\FRP}$, we obtain 
\begin{equation}\label{lem:dual-lisse-on-regular21}
    W_m(\bF_q)_{\FRP}\otimes^L_{W_n(\bF_q)} W_n\omega_{X,\log} =W_m\omega_{X,\log},
\end{equation}
which in particular yields the desired vanishing.

We now prove \ref{lem:dual-lisse-on-regular3}. Note that 
\[\HHom_{X_{\FRP}}(E_{\FRP}, W_n(\bF_q)_{\FRP})|_{X_{\et}}= \HHom_{X_{\et}}(E, W_n(\bF_q))=E^{\vee}\] 
hence by adjuntion there is a  natural map
\[(E^{\vee})_{\FRP}\to \HHom_{X_{\FRP}}(E_{\FRP}, W_n(\bF_q)_{\FRP})\]
inducing a natural map
\begin{equation}\label{lem:dual-lisse-on-regular30}
E^{\vee}(W_n\omega_{X,\log})= E^{\vee}_{\FRP}\otimes_{W_n(\bF_q)}W_n\omega_{X,\log}\to R\HHom(E_{\FRP}, W_n\omega_{X,\log}).    
\end{equation}
Showing that this map is an isomorphism is an \'etale local question, and hence we may assume that $E=W_m(\bF_q)$ for some $m\le n$.
In this case $E^\vee=W_m(\bF_q)$ and the left hand side of \eqref{lem:dual-lisse-on-regular30} becomes 
$W_m\omega_{X,\log}$ by \eqref{lem:dual-lisse-on-regular21}. 
Using the free resolution \eqref{lem:dual-lisse-on-regular20} we see that the right hand  side
is isomorphic to the complex
\[0\to W_n\omega_{X,\log}\xrightarrow{p^m} W_n\omega_{X,\log}\xrightarrow{p^{n-m}} W_n\omega_{X,\log}\xrightarrow{p^m}\ldots,\]
which by \autoref{lem:logses} is quasi-isomorphic to $W_m\omega_{X,\log}$ via the augmentation map 
$\pi^{n-m}: W_m\omega_{X,\log}\to W_n\omega_{X,\log}$. 

Finally, we prove \ref{lem:dual-lisse-on-regular4}. By \autoref{lem:reg-dualizing-log} and \ref{lem:dual-lisse-on-regular2} we have a natural map
\begin{multline*}
    E|_{Z}(W_n\omega_{Z,\log})= (E|_{Z})_{\FRP}\otimes^L W_n\omega_{Z,\log}\cong (E|_{Z})_{\FRP}\otimes^L h^!W_n\omega_{X,\log}[c]
\\ \cong h^{-1}(E_{\FRP}\otimes^L R\Gamma_Z(W_n\omega_{X,\log})[c])\to h^!E(W_n\omega_{X,\log})[c]. 
\end{multline*}
To show that is an isomorphism is an \'etale local question and hence we may assume $E=W_m(\bF_q)$, for some $m\le n$.
By \eqref{lem:dual-lisse-on-regular21},  the left hand side is isomorphic to  $W_m\omega_{Z,\log}$,  
and the right hand side is isomorphic to $h^!W_m\omega_{X,\log}[c]$.
Hence the statement holds by \autoref{lem:reg-dualizing-log}.
\end{proof}

\begin{lemma}\label{lem:dual-lisse-in-D0}
Assume $X$ is regular. Let $\cE$ be a dual-lisse  sheaf on $X$.
Then there is a coherent Cartier module $\cN$ such that $\cE=\Sol(\cN)$.
In particular $\cE\in D_0(X, W_n(\bF_q))$.    
\end{lemma}
\begin{proof}
By definition $\cE=E(W_n\omega_{X,\log})$ for some lisse sheaf $E$ on $X_{\et}$.
By the Riemann--Hilbert correspondence there is a coherent Frobenius module $\cM$ on $W_nX$ such that
$(E^\vee)_{\FRP}= \Sol(\cM)$.
Thus
\[\cE=\bD_K((E^{\vee})_{\FRP})= \bD_K(\Sol(\cM))= \Sol(\bD_G(\cM)),\]
where the first equality holds by \autoref{lem:dual-lisse-on-regular}\ref{lem:dual-lisse-on-regular3} and the last equality
by \autoref{main_commutativity_thm}. Taking cohomology and using that $\Sol$ is exact gives
\[\cE=\Sol \cH^0(\bD_G(\cM))=\Sol(\HHom_{W_n\cO_X}(\cM, W_n\omega_X)),\]
which yields the statement.
\end{proof}

\begin{corollary}\label{cor:dual-lisse-open-immersion}
Assume $X$ is regular. Let $f\colon X\to Y$ be a separated morphism with $Y$ a Noetherian and $F$-finite $\bF_q$-scheme.
Let $\cE^\bullet\in D^b(X_{\FRP}, W_n(\bF_q))$ be a bounded complex with dual-lisse cohomology sheaves. 
Then $Rf_*\cE^\bullet\in D_0(Y, W_n(\bF_q))$.
\end{corollary}
\begin{proof}
By Nagata compactification we may factor $f$ as an open immersion $j:X\inj \overline{X}$ and a proper map $\bar{f}:\overline{X}\to Y$.
As pushforward along a proper map clearly preserves $D_0$ we are reduced to the case of an open immersion $j:X\inj Y$ with
$X$ regular.
By standard d\'evissage we can assume that $\cE^\bullet=\cE$ is a  dual-lisse sheaf. 
By \autoref{lem:dual-lisse-in-D0} we find a coherent Cartier module $\cN$ on $W_nX$ such that $\cE=\Sol(\cN)$.
Hence $Rj_*\cE=\Sol(Rj_*\cN)$, by \autoref{Sol_exact_and_commutes_with_pushforwards}.
Since this is a bounded complex and $\Sol$ is exact it suffices to show that $\Sol(R^ij_*\cN)\in D_0(Y, W_n(\bF_q))$
for all $i$. But this follows from the fact that the Cartier module $R^ij_*\cN$ is locally nil-isomorphic (i.e. isomorphic up to 
locally nilpotent kernel and cokernel) to a coherent Cartier module, see \cite[Lemma 4.4.7 (b)]{Baudin_Duality_between_Witt_Cartier_crystals_and_perverse_sheaves},
as $\Sol$ sends locally nil-isomorphic $W_n(\cO_Y)$-modules to isomorphic $W_n(\bF_q)$-modules.
\end{proof}

We recall that as $X$ is Noetherian and $F$-finite it is excellent by a result of Kunz, 
see \cite[Theorem 108]{Matsumura_Commutative_Algebra}, and hence any reduced closed subscheme of $X$ has a dense 
open regular subscheme. 

\begin{definition}\label{defn:dual-constructible}
We define a {\em regular stratification} of $X$ to be a finite stratification into regular locally closed subschemes of its
reduced subscheme $X_{\red}=\coprod_{\nu=1}^e X_{\nu}$,
where $e=\dim X$ and the closure $\overline{X}_{\nu}\subset X_{\red}$ of $X_{\nu}$  has dimension $\dim \overline{X}_{\nu}=\nu$.
We denote by 
\[j^{\nu}:X^{\nu}:= X_{\red}\setminus \overline{X}_{\nu-1}\inj X_{\red}\quad \text{and}\quad 
i_{\nu}:X_\nu=\overline{X}_{\nu}\cap X^{\nu}\inj X^{\nu}\] 
the open - and closed immersion, respectively, and  the immersion of $X_{\nu}$ into $X$ by
\[h_{\nu}=\iota\circ j^{\nu}\circ i_{\nu}: X_{\nu}\inj X,\] 
where $\iota: X_{\red}\inj X$ is the closed immersion.

For a complex $\cA^\bullet\in D(X_{\FRP}, W_n(\bF_q))$ we set 
\begin{equation}\label{defn:dual-constructible1}
R\Gamma_{X_{\nu}}(\cA^\bullet):= R\Gamma_{X_{\nu}}(\cC|_{X^{\nu}})=R\Gamma_{\overline{X}_{\nu}}(\cA)|_{X^\nu}\in 
D(X^{\nu}_{\FRP}, W_n(\bF_q)).    
\end{equation}
Notice that $X_e=X^{e}$ is the open dense stratum of $X_{\red}$ and $R\Gamma_{X_e}(\cA^\bullet)=\cA^{\bullet}|_{X_e}$,
also note that $X^0=X_{\red}$ and hence $R\Gamma_{X_0}(\cA^\bullet)$ is a complex on $X_{\FRP}$, 
where we use the identification of \autoref{lem:red}.

Moreover we set (cf. \autoref{constr:upper-shriek-FRP})
\begin{equation}\label{defn:dual-constructible2}
h_\nu^!\cA^\bullet= i_{\nu}^!(\cA^\bullet|_{X^{\nu}})=i_{\nu}^{-1}R\Gamma_{X_{\nu}}(\cA^\bullet|_{X^{\nu}})
= h_{\nu}^{-1}R\Gamma_{\overline{X_{\nu}}}(\cA^\bullet).
\end{equation}
\end{definition}

Recall the category $\cC_{X,\adj}$ introduced in  \autoref{defn:C_adj}. 
It is a full triangulated subcategory of $D^+(X_{\FRP}, W_n(\bF_q))$ which in particular contains any complex of the form
$\cM^{\bullet}_{\FRP}$ with $\cM^{\bullet}\in D^+(\QCoh_{W_nX})$. We set 
\[\cC^b_{\adj}=\cC_{\adj}\cap D^b(X_{\FRP}, W_n(\bF_q)).\]

\begin{definition}\label{def_dual_constr_cpx}
A {\em dual-constructible complex} of $W_n(\bF_q)$-modules on $X$ is a bounded complex $\cF^\bullet\in \cC^b_{X,\adj}$
for which there exists a regular stratification $X_{\red}=\coprod_{\nu=0}^e X_\nu$, with $e=\dim X$ 
and  immersions $h_{\nu}:X_{\nu}\inj X$, 
such that $h_{\nu}^!(\cF^\bullet)$ has dual-lisse cohomology sheaves on $X_{\nu,\FRP}$, for all $\nu$. 
We call any such stratification a {\em regular $\cF^\bullet$-stratification}.  
\end{definition}

\begin{lemma}\label{lem:dcc-in-D0}
Let $\cF^{\bullet}$ be a dual-constructible complex.  We have
\begin{enumerate}[label=(\arabic*)]
    \item\label{lem:dcc-in-D01} $h^!\cF^{\bullet}$ is dual-constructible for any immersion $h:Z\inj X$;
    \item\label{lem:dcc-in-D02} $\cF^\bullet\in D_0(X, W_n(\bF_q))$.   
\end{enumerate}
\end{lemma}
\begin{proof}
We may assume that $X$ is reduced. Indeed, for \ref{lem:dcc-in-D01} it is equivalent to the statement for
$h_{\red}\colon Z_{\red}\inj X_{\red}$, and for \ref{lem:dcc-in-D02} if we show that $\iota^{-1}\cF^\bullet\in D_0(X_{\red},W_n(\bF_q))$,
where $\iota\colon X_{\red}\inj X$, then also $\cF=\iota_*\iota^{-1}\cF\in D_0(X, W_n(\bF_q))$. 

Let $X=\coprod_{\nu=0}^e X_{\nu}$, with $h_\nu\colon X_{\nu}\inj X$, be a regular $\cF^\bullet$-stratification. Let us prove \ref{lem:dcc-in-D01}. Let $Z=\coprod Z_{\mu}$ be a regular stratification such that any immersion $h\circ k_{\mu}\colon Z_{\mu}\inj Z\inj X$ 
factors as $h_{\nu(\mu)}\circ k_{\nu(\mu)}\colon Z_{\mu}\inj X_{\nu(\mu)}\inj X$, for some $\nu(\mu)\in\{0,\ldots, e\}$.
Then
\[k^!_{\mu} h^! \cF^{\bullet}= k_{\nu(\mu)}^!h_{\nu(\mu)}^!\cF^{\bullet}.\]
Thus there is a spectral sequence
\[E^{i,j}_2= \cH^i(k_{\nu(\mu)}^! \cH^j(h_{\nu(\mu)}^!\cF^{\bullet})) \Longrightarrow \cH^{i+j}(k^!_{\nu} h^! \cF^{\bullet}).\]
Note that $\cH^j(h_{\nu(\mu)}^!\cF^{\bullet})$ is dual-lisse by assumption. 
Thus if $c=\codim(Z_{\mu}, X_{\nu(\mu)})$,  then $E^{c,j}_2$ is dual-lisse and $E^{i,j}_2=0$ for all $i,j$ with $i\neq c$,
by \autoref{lem:dual-lisse-on-regular}\ref{lem:dual-lisse-on-regular4}. 
Hence $k^!_{\mu} h^! \cF^{\bullet}$ has dual-lisse cohomology. 
Moreover, $h^! \cF^{\bullet}\in \cC^b_{X,\adj}$, by \autoref{exact_upper_shriek_triangle}. 
Thus $h^!\cF^{\bullet}$ is dual-constructible.

Let us now prove \ref{lem:dcc-in-D02}. We use the notation from \autoref{defn:dual-constructible}.
As $\cF^{\bullet}\in \cC^b_{X,\adj}$ we have
\begin{equation}\label{lem:dcc-in-D03}
Rh_{\nu*} h_{\nu}^!\cF^{\bullet}= Rj^{\nu}_* i_{\nu*} i_{\nu}^{-1}R\Gamma_{X_{\nu}}(\cF^{\bullet}|_{X^{\nu}})
=Rj^{\nu}_*R\Gamma_{X_{\nu}}(\cF^{\bullet}|_{X^{\nu}}).
\end{equation}
As $h_{\nu}^!\cF^\bullet$ is dual-lisse, \autoref{cor:dual-lisse-open-immersion} yields that
\[Rj^{\nu}_*R\Gamma_{X_{\nu}}(\cF^{\bullet}|_{X^{\nu}})\in D_0(X, W_n(\bF_q)).\]
 Consider the exact exact triangle on $X_{\FRP}$
\begin{equation}\label{lem:dcc-in-D04}
\begin{tikzcd}
R\Gamma_{\overline{X}_{\nu-1}}(\cF^{\bullet})\arrow[r] & R\Gamma_{\overline{X}_{\nu}}(\cF^\bullet)\arrow[r] &
Rj^{\nu}_* R\Gamma_{X_{\nu}}(\cF^\bullet|_{X_{\nu}})\arrow[r, "+1"] &.    
\end{tikzcd}
\end{equation}
As  $X^0=X$ and $R\Gamma_{\overline{X}_0}(\cF^{\bullet})=R\Gamma_{X_0}(\cF^{\bullet})$
the statement holds by induction over $\nu$.
\end{proof}

\begin{proposition}\label{prop:D-h-Gamma}
Let $h:Z\inj X$ be a locally closed immersion and $\cC^{\bullet}\in D_0(X, W_n(\bF_q))$. 
Note that $h^!\cC^\bullet=h^{-1}R\Gamma_{\overline{Z}}(\cC^{\bullet})$, cf.
\autoref{defn:FRP-upper-shriek}, where $\overline{Z}\subset X$ denotes the closure of $Z$.
We have 
 \begin{enumerate}[label=(\arabic*)]
     \item\label{prop:D-h-Gamma1} 
     $h^!(\bD_{K,X}(\cC^{\bullet}))=\bD_{K,Z}(h^{-1}\cC^{\bullet});$
     \item\label{prop:D-h-Gamma2} if $h^{-1}\bD_{K,X}(\cC^{\bullet})\in D_0(X, W_n(\bF_q))$, then 
     $h^{-1}\bD_{K,X}(\cC^{\bullet})=\bD_{K,Z}(h^!\cC^{\bullet})$.
 \end{enumerate}
\end{proposition}
\begin{proof}
Let us prove \ref{prop:D-h-Gamma1} Let $i\colon \overline{Z}\inj X$ be the closed immersion and  $j\colon Z\inj \overline{Z}$ the  open immersion, so that 
$h=i\circ j$. We compute
\begin{align*}
h^!\bD_{K,X}(\cC^{\bullet}) & = h^{-1} R\Gamma_{\overline{Z}}\cR\HHom(\cC^{\bullet}, W_n\omega^\bullet_{X,\log}) & & 
\text{by defn of } \bD_{K,X} \text{ and } h^!\\
             &= h^{-1} \cR\HHom(\cC^{\bullet}, R\Gamma_{\overline{Z}} W_n\omega^\bullet_{X,\log}) & & 
\text{follows from \ref{derived_version_of_iflat}}\\
               &=h^{-1}\cR\HHom(\cC^{\bullet}, i_*i^{-1}R\Gamma_{\overline{Z}}(W_n\omega^\bullet_{X,\log})) & & 
               \text{by \ref{exact_upper_shriek_triangle}}\\
               &=h^{-1}i_*\cR\HHom(i^{-1}\cC^{\bullet}, i^!W_n\omega^\bullet_{X,\log})    & & 
       \text{by the adjunction $(i^{-1}, i_*)$}\\
              &= j^{-1}i^{-1}i_*\bD_{K,\overline{Z}}(i^{-1}\cC^{\bullet}) & & 
                           \text{by \eqref{eq:W_nomegalogZ}}\\
             &= j^{-1}\bD_{K,\overline{Z}}(i^{-1}\cC^{\bullet}) & &
               \text{by \ref{pushforward_from_closed_immersion_fully_faithful}}\\
            &= \bD_{K,Z}(h^{-1}\cC^{\bullet}). & &
\end{align*}
For \ref{prop:D-h-Gamma2} we observe 
\[\bD_{K,Z}(h^!\cC^{\bullet})= \bD_{K,Z}(h^!\bD_{K,X}(\bD_{K,X}(\cC^{\bullet})))=
\bD_{K,Z}(\bD_{K,Z}((h^{-1}\bD_{K,X}(\cC^{\bullet})))= h^{-1}\bD_{K,X}(\cC^{\bullet}),\]
where we use Kato duality for the first equality (using that $\cC^{\bullet}\in D_0(X, W_n(\bF_q))$),
\ref{prop:D-h-Gamma1} for the second equality, and again Kato duality for the third equality
(using that by assumption $h^{-1}\bD_{K,X}(\cC^{\bullet})\in D_0(X, W_n(\bF_q))$).
\end{proof}

\begin{definition}\label{def:Ddc}
Let  $D^b_c(X_{\et}, W_n(\bF_q))\subset D^b(X_{\et}, W_n(\bF_q))$ be the full triangulated subcategory 
formed by bounded complexes with constructible cohomology.
We define the category $D^b_c(X_{\FRP}, W_n(\bF_q))$ to be its essential image under the exact and fully faithful functor 
\[(-)_{\FRP}\colon D^b(X_{\et}, W_n(\bF_q))\to D_0(X, W_n(\bF_q)),\] 
see \autoref{adjunction_FRP_and_etale}\ref{FRP_embeds_in_etale} and \autoref{def:D0}.
Here we use \autoref{Sol_on_Frobenius_modules_comes_from_etale_site} and the Riemann--Hilbert correspondence
for Frobenius crystals (\autoref{Riemann_Hilbert_Frobenius_modules}) to see that the essential image of the above functor is indeed contained in $D_0(X, W_n(\bF_q))$.

We define $D^{b}_{dc}(X_{\FRP}, W_n(\bF_q))$ to be the full subcategory of $D_0(X, W_n(\bF_q))$
consisting of dual-constructible complexes on $X$.
\end{definition}

\begin{proposition}\label{prop:Ddc-is-c}
Let $\cF^\bullet\in D^b_{dc}(X_{\FRP}, W_n(\bF_q))$ be a dual-constructible complex on $X$.
Then $\bD_{K,X}(\cF^{\bullet})\in D^b_c(X_{\FRP}, W_n(\bF_q))$.
\end{proposition}
\begin{proof}
Let $\iota: X_{\red}\inj X$ be the closed embedding. Note that $i^{-1}\cF^\bullet=i^!\cF^{\bullet}$ is dual-constructible.
If the complex  $\bD_{K, X_{\red}}(i^{-1}\cF^{\bullet})$ has constructible cohomology, then so does the complex
\[\iota_*\bD_{K, X_{\red}}(i^{-1}\cF^{\bullet})=\bD_{K,X}(\iota_*\iota^{-1}\cF^{\bullet})=\bD_{K,X}(\cF^{\bullet}).\]
Thus we may assume $X$ to be reduced.

Let $X=\coprod_{\nu=0}^e X_{\nu}$, with $h_{\nu}:X_{\nu}\to X$, be a regular $\cF^{\bullet}$-stratification.
We use the notation from \autoref{defn:dual-constructible}. As $D^b_c(X_{\FRP}, W_n(\bF_q))$ is triangulated
it suffices by the triangle \eqref{lem:dcc-in-D04} and the identification \eqref{lem:dcc-in-D03} to show that 
the bounded complex $\bD_{K,X}(Rj^{\nu}_*h_{\nu}^!\cF^{\bullet})$ has constructible cohomology for all $\nu$.
By considering the triangle
\[
\begin{tikzcd}
\tau_{\leq i-1}h_{\nu}^!\cF^{\bullet}\arrow[r] & \tau_{\leq i} h_{\nu}^!\cF^{\bullet}\arrow[r] & \cH^i(\cF^{\bullet})[-i]\arrow[r] &   , 
\end{tikzcd}
\]
we are reduced to show that $\bD_{K,X}(Rj^{\nu}_*\cE)$
has constructible cohomology for all $\nu$, with $\cE=\cH^i(h_{\nu}^!\cF^{\bullet})$  dual-lisse on $X_{\nu}$.
A further such argument reduces  us to show that
$\bD_{K,X}(R^ij^{\nu}_*\cE)$
has constructible cohomology for all $i$, $\nu$, and all dual-lisse sheaves $\cE$ on $X_{\nu}$.
By \autoref{lem:dual-lisse-in-D0} there is coherent Cartier module $\cN$ on $X_{\nu}$ such that $\cE=\Sol(\cN)$ 
and hence the argument at the end of the proof of \autoref{cor:dual-lisse-open-immersion} (relying on
\cite[Lemma 4.4.7(b)]{Baudin_Duality_between_Witt_Cartier_crystals_and_perverse_sheaves})
shows that there is a coherent Cartier module $\overline{\cN}$ on $X$ such that
$R^i j^{\nu}_*\cE= \Sol(\overline{\cN})$. 
Thus by \autoref{main_commutativity_thm}
\[\bD_{K,X}(R^ij^{\nu}_*\cE)= \bD_{K,X}(\Sol(\overline{\cN}))= \Sol(\bD_{G,X}(\overline{\cN}))\]
equals the solutions of a complex with  $r$-Frobenius crystals as cohomology groups,
hence it has constructible cohomology groups by the classical Riemann--Hilbert correspondence for Frobenius modules, see \autoref{Riemann_Hilbert_Frobenius_modules}.
\end{proof}

\begin{proposition}\label{prop:Dc-is-dc}
Let $\cG^\bullet\in D^b_{c}(X_{\FRP}, W_n(\bF_q))$ be a complex with constructible cohomology on $X$.
Then $\bD_K(\cG^{\bullet})\in D^b_{dc}(X_{\FRP}, W_n(\bF_q))$.
\end{proposition}
\begin{proof}
We may assume $X$ reduced and by d\'evissage that $\cG^{\bullet}=\cG_{\FRP}$ with $\cG$ a constructible sheaf on $X_{\et}$.
As $X$ is excellent we find 
a regular stratification $X=\coprod_{\nu=0}^e X_{\nu}$, with $h_{\nu}: X_{\nu}\inj X$, such that 
$E=h^{-1}_{\nu}\cG$ is a lisse sheaf on $X_{\nu,\et}$. 
Thus the statement follows from \autoref{prop:D-h-Gamma}\ref{prop:D-h-Gamma1}, together with the fact that 
$\cG_{\FRP}\in D_0(X, W_n(\bF_q))$ by the classical Riemann--Hilbert correspondence, and 
\autoref{lem:dual-lisse-on-regular}\ref{lem:dual-lisse-on-regular3}.
\end{proof}

\begin{thm}[{Second half of \autoref{intro_commutativity}}]\label{thm:constr-dual}
Let $X$ be a Noetherian and $F$-finite $\bF_q$-scheme. Then Kato's dualizing functor induces an equivalence of categories
\[\begin{tikzcd}
\bD_{K,X} \colon D^b_c(X_{\et}, W_n(\bF_q))^{\rm op}\cong D^b_c(X_{\FRP}, W_n(\bF_q))^{\rm op}\arrow[r, "\simeq"] &
D^b_{dc}(X_{\FRP}, W_n(\bF_q)).    
\end{tikzcd}
\]
Moreover, for any locally closed immersion $h:Z\inj X$ we have 
\begin{enumerate}[label=(\arabic*)]
    \item\label{thm:constr-dual1} $h^{-1}\bD_{K,X}(\cF^{\bullet})= \bD_{K,Z}(h^!\cF^{\bullet})$, 
    for all $\cF^{\bullet}\in D^b_{dc}(X_{\FRP}, W_n(\bF_q))$;
    \item\label{thm:constr-dual2} $h^!\bD_{K,X}(\cG^{\bullet})= \bD_{K,Z}(h^{-1}\cG^{\bullet})$, for all
    $\cG^{\bullet}\in D^b_c(X_{\FRP}, W_n(\bF_q))$.
\end{enumerate}
\end{thm}
\begin{proof}
The first statement follows from Kato's duality theorem \autoref{Kato_duality} together with
\autoref{prop:Ddc-is-c} and \autoref{prop:Dc-is-dc}.
The second part of the statement is a special case of \autoref{prop:D-h-Gamma}.
\end{proof}

\begin{corollary}[Second half of {\autoref{main_thm}}]\label{cor:Sol-CartierCrys}
The solution functor induces an equivalence of categories
\[\begin{tikzcd}
\Sol\colon D^b(\Crys^{C^r}_{W_nX})\arrow[r,"\simeq"] & D^b_{dc}(X_{\FRP}, W_n(\bF_q)).    
\end{tikzcd}
\]
\end{corollary}
\begin{proof}
    This follows from the classical Riemann--Hilbert correspondence in the version of 
    \cite[Theorem 2.2.14 and Proposition 2.2.15]{Baudin_Duality_between_Witt_Cartier_crystals_and_perverse_sheaves}, together
    with \autoref{main_commutativity_thm}, and \autoref{thm:constr-dual}. 
\end{proof}

\begin{corollary}\label{cor:pf-dc}
Let $f\colon Y\to X$ be separated morphism of finite type between Noetherian $F$-finite $\bF_q$-schemes. 
Then the derived pushforward induces a well-defined functor
\[Rf_*\colon D^b_{dc}(Y_{\FRP}, W_n(\bF_q))\to D^b_{dc}(X_{\FRP}, W_n(\bF_q)).\]
\end{corollary}
\begin{proof}
Note that $Y$ has a unit dualizing complex, namely $f^!W_n\omega^\bullet_X$, so $D^b_{dc}(Y_{\FRP}, W_n(\bF_q))$ is well-defined.
Now the statement follows from the fact that $Rf_*$ induces a well-defined functor 
\[Rf_*\colon D^b(\Crys^{C^r}_{W_nY})\to D^b(\Crys^{C^r}_{W_nX}),\]
by \cite[Proposition 4.4.4]{Baudin_Duality_between_Witt_Cartier_crystals_and_perverse_sheaves},
and commutes with $\Sol$ by  \autoref{Sol_exact_and_commutes_with_pushforwards}.
\end{proof}

\begin{corollary}
Consider the following cartesian diagram
\[
\begin{tikzcd}
Y_Z\arrow[r, "k"]\arrow[d, "g"] & Y\arrow[d, "f"]\\
Z\arrow[r, "h"] & X,
\end{tikzcd}
\]
in which $f$ is proper and $h\colon Z\inj X$ an immersion.
Let $\cF^\bullet$ be a dual-constructible complex on $X$. Then there is a natural isomorphism
\[Rg_* k^!\cF^{\bullet} \cong h^!Rf_*\cF^{\bullet}.\]
\end{corollary}
\begin{proof}
    Thanks to \autoref{exact_upper_shriek_triangle}\ref{exact_upper_shriek_triangle1}, we have by adjunction a natural transformation \[ Rg_*k^! \to h^!Rf_* \] on $D_0(Z, W_n(\bF_q))$. Let us show that it is an isomorphism on dual-constructible complexes. This is local on $X$ so we may assume that all schemes involved are separated. Write $\cF^{\bullet} \cong \Sol(\cM^{\bullet})$ for $\cM^{\bullet} \in D^b(\Crys^{C^r}_{W_nX})$, see \autoref{cor:Sol-CartierCrys}. By \autoref{Sol_commutes_with_upper_shriek} and \autoref{Sol_exact_and_commutes_with_pushforwards}, it is enough to show the analogous statement on Cartier crystals on $W_nX$. This follows from the same proof as that of \cite[Corollary 5.3.5]{Baudin_Duality_between_perverse_sheaves_and_Cartier_crystals}, using the results of \cite{Baudin_Duality_between_Witt_Cartier_crystals_and_perverse_sheaves}. \qedhere
    
    %In this case, we have a natural isomorphism \[ Rf_* \circ \bD_G \cong \bD_G \circ Rf_* \colon \] have 
    %\[\begin{tikzcd}
     % Rf_*\bD_{K,Y}= \bD_{K,X} Rf_*\colon   D^b_{\rm c}(X_{\et}, W_n(\bF_q))\arrow[r] & D^b_{\rm dc}(X_{\FRP}, W_n(\bF_q)),
    %\end{tikzcd}\]
    %as follows from \JB{I rewrite this part} \autoref{main_commutativity_thm}, the fact that $\bD_{G,X} Rf_*=Rf_* \bD_{G,Y}$
    %on $D^+_{\coh}(\QCoh_{W_nY})$ by Grothendieck duality, and the fact that $\Sol$ commutes with $Rf_*$
   % on $D^+(\IndCoh^{C^r}_{W_nY})$ by \autoref{Sol_exact_and_commutes_with_pushforwards}.\footnote{For schemes locally 
    %of finite type over a perfect field, this also follows from \cite[Theorem (0.2)]{Kato_Duality_theories_for_p_primary_etale_coh_II}. } Thus the statement follows from \autoref{thm:constr-dual}
    %and proper base change for \'etale torsion sheaves.    
\end{proof}

\begin{corollary}\label{cor:pf-log}
Let $f\colon Y\to X$ be a separated morphism between Noetherian $F$-finite $\bF_q$-schemes. 
Let $\cF^\bullet$ be any dual-constructible complex on $Y$.
Then there exists a regular dense open subscheme $U\subset X$ and a lisse sheaf of $W_n(\bF_q)$-modules $E$ on $U$ 
such that 
\[(R^if_* \cF^{\bullet})|_{U_{\et}}= E\otimes_{W_n(\bF_q)} W_n\Omega^e_{U_{\et},\log},\]
where $e={\rm rk}(\Omega^1_{U/\bF_p})$. 
\end{corollary}
\begin{proof}
On $X_{\FRP}$ this follows from \autoref{cor:pf-dc} and \autoref{lem:dual-regular-case}.
Then restricting to $X_{\et}$ gives the result as  follows from \autoref{Sol_on_Frobenius_modules_comes_from_etale_site},
\autoref{cor:Sol-CartierCrys}, and the fact that for a coherent sheaf $\cM$ on $X$ we have 
$Rf_* (\cM_{\FRP})|_{X_{\et}}= Rf_*(\cM_{\et})$.
\end{proof}

\begin{corollary}[{\autoref{intro_cor:sep-base}}]\label{cor:sep-base}
Let $K$ be a separably closed field which has a $p$-basis of length $e$.
Let $X$ be a $d$-dimensional regular connected separated $K$-scheme.
Then for all $i$ 
\[H^i(X_{\et}, W_n\Omega^{d+e}_{X_{\et},\log})= M_i\otimes_{\bZ/p^n\bZ} W_n\Omega^e_{K,\log},\]
for some finite $\bZ/p^n\bZ$-module $M_i$.    
\end{corollary}
\begin{proof}
Apply  \autoref{cor:pf-log} in the case $r=1$ and $\cF^{\bullet}=W_n\Omega^{d+e}_{X_{\et},\log}$.
\end{proof}

\begin{corollary}\label{cor:pf-Geisser-cyc}
Let $f\colon X\to Y$  be a separated morphism between schemes of finite type over a perfect field $k$ of dimension $d$ and $e$, respectively.
Let $\bZ^c_X(0)/p^n$ be Geisser's \'etale cycle complex, see \cite[2.]{Geisser_Duality_via_cycle_complexes}
(see also \autoref{rmk:dualizing-dRW}). 
Then there exists a smooth open dense  subscheme $U\subset Y$ and a lisse sheaf of $\bZ/p^n\bZ$-modules $E$ on $U$ 
such that 
\[(R^if_* \bZ^c_X(0)/p^n)|_{U_{\et}} \cong E\otimes_{\bZ/p^n\bZ} \bZ^c_U(0)/p^n[-e].\]
\end{corollary}
\begin{proof}
Using the isomorphism  $W_n\omega^{\bullet}_{X,\log}|_{X_{\et}}\cong \bZ^c_X(0)/p^n$ from \cite[Theorem 6.1]{Ren_Bloch_cycle_complex_coherent_dualizing_complexes}, see \autoref{rmk:dualizing-dRW}, 
we can argue as in \autoref{cor:pf-log}.
\end{proof}

\bibliographystyle{amsalpha} 
\bibliography{bibliography}
\end{document}